\documentclass[10pt,reqno]{amsart}
\usepackage{amsmath}
\usepackage{amssymb}
\usepackage{amsthm}
\usepackage[usenames]{color}
\usepackage{eepic,epic}

\newtheorem{thm}{Theorem}[section]
\newtheorem{cor}[thm]{Corollary}
\newtheorem{lem}[thm]{Lemma}
\newtheorem{prop}[thm]{Proposition}

\newtheorem{theoremalpha}{Theorem}

\theoremstyle{definition}

\theoremstyle{remark}
\newtheorem{rem}[thm]{Remark}
\numberwithin{equation}{section}

\begin{document}
\title[The splitting method for the nonlinear Schr\"odinger equation]
{On the splitting method for the nonlinear Schr\"odinger equation with initial data in $H^1$ }
\author[Choi]{Woocheol Choi}
\address{Department of Mathematics Education, Incheon National University, Incheon 22012, Republic of Korea}
\email{choiwc@inu.ac.kr}

\author[Koh]{Youngwoo Koh}
\address{Department of Mathematics Education, Kongju National University, Kongju 32588, Republic of Korea}
\email{ywkoh@kongju.ac.kr}

\maketitle
\begin{abstract}
In this paper, we establish a convergence result for the operator splitting scheme $Z_{\tau}$ introduced by Ignat \cite{I}, with initial data in $H^1$, for the nonlinear Schr\"odinger equation:
    $$
    \partial_t u = i \Delta u + i\lambda |u|^{p} u,\qquad u (x,0) =\phi (x),
    $$
where $p >0$, $\lambda \in \{-1,1\}$ and $(x,t) \in \mathbb{R}^d \times [0,\infty)$.
We prove the $L^2$ convergence of order $\mathcal{O}(\tau^{1/2})$ for the scheme with initial data in the space $H^1 (\mathbb{R}^d)$ for the energy-subcritical range of $p$.
\end{abstract}


\section{Introduction}

Consider the following Cauchy problem of the nonlinear Schr\"odinger equation in $\mathbb{R}^{d+1}$:
    \begin{equation}\label{eq-main}
    \left\{\begin{array}{ll}
    \partial_t u = i \Delta u + i \lambda |u|^p u, \quad (x,t) \in \mathbb{R}^d \times \mathbb{R}, \\
    u(x,0)= \phi (x),
    \end{array} \right.
    \end{equation}
where $\lambda \in \{-1,1\}$. Nonlinear Schr\"odinger equations appear in various models of quantum mechanics (see, e.g., \cite{Ca, SS, T}). In this paper, we are concerned with operator splitting schemes, which are useful for the numerical computation of semilinear-type equations \eqref{eq-main}. The idea of such schemes is to divide the problem \eqref{eq-main} into a linear flow and a nonlinear flow, as described below.

We define $N(t) \phi$ as the solution of the flow
    $$
    \left\{ \begin{aligned}  \partial_t u & = i \lambda |u|^p u,& x \in \mathbb{R}^d,~ t >0,
    \\
    u(x,0)&= \phi (x),& x \in \mathbb{R}^d,
    \end{aligned}
    \right.
    $$
that is, $N(t) \phi = \exp (i t \lambda |\phi|^p) \phi$. On the other hand, we set $S(t) \phi$ as the solution of the linear Schr\"odinger propagation
    $$
    \left\{ \begin{aligned}  \partial_t u & = i\Delta u, & x \in \mathbb{R}^d,~ t >0,
    \\
    u(x,0)&= \phi (x),& x \in \mathbb{R}^d,
    \end{aligned}
    \right.
    $$
which admits the Fourier multiplier formula $S(t) \phi = e^{it \Delta} \phi$. Then we split the flow of \eqref{eq-main} into the flows $N(t)$ and $S(t)$ with a small switching time. Namely, for a fixed time interval $[0,T]$ and a small value $\tau>0$, we can consider the Lie approximation
    $$
    Z(n \tau) \phi = \big( S(\tau) N(\tau) \big)^n \phi,\quad 0 \leq n \tau \leq T,
    $$
or the Strang approximation
    $$
    Z(n \tau) \phi = \big( S(\tau/2) N(\tau)S(\tau/2) \big)^n \phi,\quad 0 \leq n \tau \leq T.
    $$
The convergence of these two schemes has been studied by Besse \emph{et al}. \cite{BBD} for globally Lipschitz continuous nonlinearities and by Lubich \cite{L} for Schr\"odinger-Poisson and cubic NLS equations with initial data in the space $H^4 (\mathbb{R}^3)$. On the other hand, Ignat and Zuazua \cite{IZ1, IZ2} and Ignat \cite{I2} developed various numerical schemes for which they proved Strichartz type estimates to obtain the convergence of the schemes with initial data of low regularity. Also, Ignat \cite{I} introduced the following modified version of the splitting scheme:
    \begin{equation}\label{eq-r-30}
    Z_{\tau} (n\tau) = \big( S_{\tau} (\tau) N(\tau) \big)^n \Pi_{\tau}\phi.
    \end{equation}
Here, $S_{\tau}(t)$ denotes the frequency localized Schr\"odinger flow given by
    $$
    S_{\tau}(t) \phi = S(t) \Pi_{\tau} \phi,
    $$
where
    \begin{equation}\label{eq-1-10}
    \widehat{\Pi_{\tau} \phi} (\xi)
    = \chi (\tau^{1/2}\xi) \widehat{\phi}(\xi) ,\quad \xi \in \mathbb{R}^d,
    \end{equation}
and $\chi \in C^{N}(\mathbb{R}^d)$ is a cut-off function supported in $B^d (0,2)$ such that $\chi \equiv 1$ on $B^d (0,1)$, where $N \in \mathbb{N}$ is some large number.  In fact, it is sufficient to set $N=2d$.
\

The aim of this paper is to determine an improved estimate for the splitting scheme $Z_{\tau}(n\tau)$. In particular, we prove a convergence result for $p$ in the energy-subcritical range when the initial data $\phi$ belongs to the space $H^1 (\mathbb{R}^d)$. Before stating our result, we recall some previous results.
\begin{itemize}
\item (Lubich \cite{L})
Let $d=3$ and $p=2$. Suppose that $\phi \in H^4 (\mathbb{R}^3)$, and consider a time $T>0$ such that $\sup_{0 \leq t \leq T} \|u(t)\|_{H^4 (\mathbb{R}^3)} < \infty$. Then,
the approximation $Z$ satisfies
    $$
    \max_{0 \leq n \tau \leq T} \|Z(n \tau)- u (n\tau)\|_{L^2 (\mathbb{R}^3)} \leq \tau^{2}  C(T, \phi)
    $$
and
    $$
    \max_{0 \leq n \tau \leq T} \|Z(n \tau)- u (n\tau)\|_{H^2 (\mathbb{R}^3)} \leq \tau\,  C(T,\phi).
    $$
\item (Ignat \cite{I})
Let $1\leq d\leq3$ and $1\leq p < \frac{4}{d}$. For any $\phi \in H^2 (\mathbb{R}^d)$ and any time $T>0$,
the approximation $Z_\tau$ satisfies
    $$
    \max_{0 \leq n \tau \leq T} \|Z_{\tau}(n \tau)- u (n\tau)\|_{L^2 (\mathbb{R}^d)} \leq \tau\,  C(d,p,T,\|\phi\|_{H^2}).
    $$
\end{itemize}
%
In the following result, we provide the convergence result for initial data in $H^1 (\mathbb{R}^d)$.
\begin{thm}\label{thm-2}
Let $d \geq 1$ and $0< p < \frac{4}{d}$. For any $\phi \in H^1 (\mathbb{R}^d)$ and any time $T>0$,
the approximation $Z_\tau$ satisfies
    $$
    \max_{0 \leq n \tau \leq T} \big\| Z_{\tau}(n \tau)- u (n\tau) \big\|_{L^2 (\mathbb{R}^d)} \leq \tau^{1/2}  C(d,p,T, \|\phi\|_{H^1})
    $$
for any $\tau \in (0,1)$.
\end{thm}

\begin{thm}\label{thm-2'}
Let $1\leq d <6$ and $1\leq p < p_d$.
Suppose that $\phi \in H^1 (\mathbb{R}^d)$, and consider a time $T>0$ such that $\sup_{0 \leq t \leq T} \|u(t)\|_{H^1 (\mathbb{R}^d)} < \infty$.
Then, the approximation $Z_\tau$ satisfies
    $$
    \max_{0 \leq n \tau \leq T} \big\| Z_{\tau}(n \tau)- u (n\tau) \big\|_{L^2 (\mathbb{R}^d)} \leq \tau^{1/2} C(d,p,T, \phi)
    $$
for any $\tau \in (0,1)$. Here, $p_d=\infty$ if $d=1,2$ and $p_d= \frac{4}{d-2}$ if $d\geq3$.
\end{thm}

\begin{rem}
In fact, we can find an upper bound of $C(d,p,T, \|\phi\|_{H^1})$ in Theorem \ref{thm-2} such as
    $$
    C(d,p,T, \|\phi\|_{H^1})
    \leq \exp\Big( \exp\big( C_{d,p} T^{c_1(d,p)} \|\phi\|_{H^1}^{c_2(d,p)} \big) \Big)
    $$
for some constants $C_{d,p}$, $c_1(d,p)$, $c_2(d,p)>0$.
\end{rem}

In order to obtain a convergence result with the low regularity assumption, Strichartz-type estimates are employed in \cite{I} along with the Duhamel-type formula for $Z_{\tau}$, given by
    \begin{equation}\label{eq-1-9}
    Z_{\tau}(n \tau) = S_{\tau}(n \tau) \phi + \tau \sum_{k=0}^{n-1} S_{\tau}(n\tau -k \tau) \frac{N(\tau) - I}{\tau} Z_{\tau}(k \tau),\quad n \geq 1,
    \end{equation}
which is compared to the Duhamel formula of the solution $u$ to \eqref{eq-main}, expressed as
    \begin{equation}\label{eq-1-8}
    u(t) = S(t) \phi + i \lambda \int_0^t S(t-s) |u|^p u (s) ds,\quad t \geq 0.
    \end{equation}
A key ingredient of the convergence analysis in \cite{I} is to obtain the stability (uniformly in $\tau \in (0,1)$) of the scheme $Z_{\tau}$ in the discrete space
$\ell^q (n \tau \in I;\, L^r (\mathbb{R}^d))$.

Here, we introduce a few notations. For any interval $I \subset [0,\infty)$, we define the space $\ell^q (n \tau \in I;\, L^r (\mathbb{R}^d))$ as
consisting of functions defined on $\tau \mathbb{Z} \cap I$ with values in $L^r (\mathbb{R}^d)$, the norm of which is given by
    $$
    \|u\|_{\ell^q (n\tau \in I;\, L^r (\mathbb{R}^d))}
    = \bigg( \tau \sum_{n\tau \in I} \|u(n \tau)\|_{L^r (\mathbb{R}^d)}^q \bigg)^{1/q}.
    $$
In the present work, we take into account the nonlinearity in the energy-subcritical range: $|u|^{p} u$ for $p \in (0, p_d)$, where $p_{d}$ is defined by
    $$
    p_{d} = \biggl\{\begin{array}{ll} \frac{4}{d-2} &\quad \textrm{if}~ d\geq3
    \\
    \infty &\quad \textrm{if}~d =1,2.
    \end{array}\biggr.
    $$
Also, a pair $(q,r) \in [2, \infty] \times [2,\infty]$ is called an admissible pair if
    $$
    \frac{2}{q} + \frac{d}{r} = \frac{d}{2}, \quad (q,r,d)\neq(2,\infty,2).
    $$
Lastly, we always denote $(q_0, r_0)$ be the admissible pair $(q_0, r_0)=(\frac{4(p+2)}{dp}, p+2)$.\\

The well-posedness theory on \eqref{eq-main} for $\phi \in H^1 (\mathbb{R}^d)$ is well understood as follows.

\begin{theoremalpha}\label{wp_u_H12}
[See, e.g., \cite{Ca}.] Let $d \geq 1$ and $0 < p < p_{d}$, and suppose that $\phi \in H^1 (\mathbb{R}^d)$. Then, there is a time $T_{max}=T(d,p,\phi) \in (0,\infty]$ such that a solution $u \in C\left([0,T_{max}); H^1 (\mathbb{R}^d)\right)$ to \eqref{eq-main} exists in the sense of the Duhamel formula \eqref{eq-1-8}.
Moreover, for any $T<T_{max}$ and any admissible pairs $(q,r)$, there is a positive constant $M_1 =M_1(d,p,T,\phi)>0$ such that
    \begin{equation}\label{def_C_1}
    \|u\|_{L^\infty ([0,T]; H^1)} + \|u\|_{L^q ([0,T]; W^{1,r})}
    \leq M_1.
    \end{equation}
In addition, one of the following is true:
    \begin{itemize}
    \item The solution $u$ exists globally, i.e., $T_{max}=\infty$ and $\sup_{t \in [0,\infty)} \|u(t)\|_{H^1 (\mathbb{R}^d)} < \infty.$
    \item The solution $u \in C\left([0,T_{max});\, H^1 (\mathbb{R}^d)\right)$ exists for a maximal time interval $[0, T_{max})$, and
        $$
        \lim_{t \rightarrow T_{max}} \|u(t)\|_{H^1 (\mathbb{R}^d)} = \infty.
        $$
    \end{itemize}
\end{theoremalpha}
It is well known that if $0<p<\frac{4}{d}$, then we always have $T_{max} = \infty$, due to the mass conservation property. In this case, we can write $M_1 = C_{d,p} \|\phi\|_{H^1}$ in \eqref{def_C_1}.
Further, if the equation is defocusing, i.e., $\lambda =-1$ in \eqref{eq-main}, then the solution $u$ exists globally for $0 <p< p_d$ by the energy conservation law. We refer the reader to Sections 4 and 5 of Cazenave \cite{Ca} for the details.\\

Here, the theorem below is the one of main contribution of this paper.

\begin{thm}\label{thm-rel}
Let $d \geq 1$ and $0 < p < p_{d}$. Suppose that $\phi \in H^1 (\mathbb{R}^d)$, and consider a time $T>0$ such that $\sup_{0\leq t \leq T} \|u(t)\|_{H^1 (\mathbb{R}^d)} < \infty$.
Suppose that there is a constant $M_2 = C(d,p,T,\phi) >0$ such that the stability of $Z_\tau$
    \begin{equation}\label{eq-5-51}
    \| Z_{\tau}(n \tau)\|_{\ell^\infty (n\tau \in [0,T]; H^1 )} +
    \| Z_{\tau}(n \tau)\|_{\ell^{q_0} (n\tau \in [0,T]; W^{1,r_0} )}
    \leq M_2
    \end{equation}
holds for all $\tau \in (0,1)$.
Then we have the $L^2$ convergence between $u$ with $Z_\tau$
    \begin{equation}\label{rel-goal}
    \max_{0 \leq n \tau \leq T} \big\| Z_{\tau}(n \tau)- u (n\tau) \big\|_{L^2 (\mathbb{R}^d)} \leq \tau^{1/2} \exp\Big( C_{d,p} T \big( M_1 + M_2 \big)^{\frac{2p(p+2)}{4-(d-2)p}} \Big).
    \end{equation}
In the above, the constant $M_1 >0$ in \eqref{rel-goal} is referred to in \eqref{def_C_1}.
\end{thm}

By Theorem \ref{thm-rel}, it is enough to prove the stability \eqref{eq-5-51} of $Z_\tau$ in order to show the convergence of $Z_\tau$.  
Namely, to prove Theorem \ref{thm-2} and Theorem \ref{thm-2'}, it is enough to obtain the global\,(-in-time) stability of $Z_\tau$ with $\phi \in H^1 (\mathbb{R}^d)$ in the space $\ell^{q} (n\tau \in [0,T]; W^{1,r} (\mathbb{R}^d))$ for any $T < T_{max}$. When $0< p <\frac{4}{d}$, the stability result on $Z_{\tau}$ with $\phi \in L^2 (\mathbb{R}^d)$ in the space $\ell^{q}(n\tau \in [0,T]; L^r (\mathbb{R}^d))$ was obtained in \cite[Theorem 1.1]{I} for every admissible pair $(q,r)$. A crucial observation in the proof of Theorem \ref{thm-9} is that the scheme $Z_{\tau}$ does not increase its $L^2 (\mathbb{R}^d)$ norm. On the other hand, the $H^1 (\mathbb{R}^d)$ norm of $Z_{\tau}$ is not guaranteed that it will not increase. This is a main difficulty when we obtain the desired global(-in-time) stability in Theorem \ref{thm-9'}.
\begin{thm}\label{thm-9}
Let $d \geq 1$ and $0<p<\frac{4}{d}$, and suppose that $\phi \in H^1 (\mathbb{R}^d)$.
Then, for any time $T>0$ and any admissible pair $(q,r)$, the approximation $Z_\tau$ satisfies
    $$
    \| Z_{\tau}(n \tau)\|_{\ell^q (n\tau \in [0,T]; W^{1,r} )}
    \leq  \exp\Big( C_{d,p}T  \max\big\{ \|\phi\|_{H^1}^{\frac{2p(p+2)}{4-(d-2)p}}, ~\|\phi\|_{L^2}^{\frac{4p}{4-dp}}\big\} \Big)
    $$
for any $\tau \in (0,1)$.
\end{thm}

\begin{thm}\label{thm-9'}
Let $1\leq d <6$ and $1\leq p < p_d$.
Suppose that $\phi \in H^1 (\mathbb{R}^d)$, and consider a time $T>0$ such that $\sup_{0 \leq t \leq T} \|u(t)\|_{H^1 (\mathbb{R}^d)} < \infty$.
Then, for any admissible pair $(q,r)$, there is a constant $C(d,p,T,\phi)>0$ such that the approximation $Z_\tau$ satisfies
    $$
    \| Z_{\tau}(n \tau)\|_{\ell^q (n\tau \in [0,T]; W^{1,r} )}
    \leq  C(d,p,T,\phi)
    $$
for any $\tau \in (0,1)$.
\end{thm}

Towards this global $H^1$ stability result, we first prove the corresponding local\,(-in-time) stability result (see Proposition \ref{thm-3}). Then for any $T>0$, we extend the local $H^1$ stability onto $[0,T]$ by an induction after dividing $[0,T]$ into small subintervals. This inductive step is provided separately for the cases $0 < p< \frac{4}{d}$ and $1 \leq p <p_d$. In the case that $0<p <\frac{4}{d}$, we obtain an inductive estimate for proving \mbox{Theorem \ref{thm-9}} by combining the local $H^1$ stability with the $\ell^{q} (n\tau \in [0,T]; L^{r} (\mathbb{R}^d))$ stability result on $Z_{\tau}$, obtained by Ignat \cite{I}. In the estimate, we will see that one may extend the local stability of $Z_{\tau}$ to the Sobolev space $\ell^{q} (n\tau \in I; W^{1,r} (\mathbb{R}^d))$ on the next interval $I$ once the value $|I|^{1-\frac{dp}{4}}\|Z_{\tau}(n\tau)\|_{\ell^{q} (n\tau \in I; L^{r})}$ is small enough.

This procedure breaks down in the mass super-critical case $\frac{4}{d}\leq p <p_d$, since we do not have \emph{a priori} bound on $\|Z_{\tau}(n\tau)\|_{\ell^{q} (n\tau \in (0,T); L^{r})}$ with some  $(q,r)$. Instead, we shall apply the local $H^1$ stability result of Proposition \ref{thm-3} recursively in a direct way. To make it possible, the major task is to control the growth of the $\|Z_{\tau}\|_{H^1 (\mathbb{R}^d)}$ norm when iterating the local $H^1$ stability result. For this aim, we turn to verify that the scheme $Z_{\tau}$ converges to the solution $u$ in $H^1 (\mathbb{R}^d)$ as $\tau>0$ goes to $0^{+}$ on an interval where the $H^1$ stability of $Z_{\tau}$ is known. It then enables us to utilize the fact that $\|u(t)\|_{H^1}$ is bounded on the interval $[0,T]$ for any fixed $T< T_{max}$. Consequently, we have a good control on $\|Z_{\tau}\|_{H^1 (\mathbb{R}^d)}$ when $\tau >0$ is small enough. By exploiting this idea, we will obtain the global $H^1$ stability by iterating the local $H^1$ stability for the case $\tau \in (0,T_*)$ with a suitable choice of $T_* = T_*(d,p,T,\phi)>0$.
\


The remainder of this paper is organized as follows.
In Section \ref{sec-2}, we recall the Strichartz estimates and their discrete versions of these for the modified linear flow $S_{\tau}(t)$. In addition, we present some estimates for $N(\tau)$ and $\Pi_{\tau}$, and recall a detail regarding the well-posedness result for \eqref{eq-main}.
In Section \ref{sec-6}, we present the proof of Theorem \ref{thm-2} for the energy-subcritical case $0< p < p_d$.
In Section \ref{sec-4}, we prove the local $H^1$ stability result on the splitting scheme $Z_{\tau}$, and we prove Theorem \ref{thm-9} which is the mass-subcritical case $0< p < \frac{4}{d}$.
In Section \ref{sec-10'}-\ref{sec-10}, we prove Theorem \ref{thm-9'} which is the energy-subcritical case $1\leq p < p_d$.

\

\noindent \textbf{Notations}

\begin{itemize}
\item If a constant depends on some other values, we mark it like as $C_T$ (depending on time $T$) or $C(d,p,T,\phi)$ (depending on dimension $d$, nonlinear exponent $p$, time  $T$ and initial data $\phi$). We also use the notations $\alpha_{d,p}$, $\beta_{d,p}$, and $\gamma_{d,p}$ to denote positive constants determined by $d$ and $p$.

\item For $0  \leq a<b < \infty$, we often write $\|\cdot\|_{\ell^{q} (n\tau \in [a,b]; \mathcal{B})}$ as $\|\cdot\|_{\ell^q (a,b;\mathcal{B})}$ for $\mathcal{B} = W^{k,q} (\mathbb{R}^d)$ or $L^q (\mathbb{R}^d)$.

\item We often simply denote $W^{k,q} (\mathbb{R}^d)$ as $W^{k,q}$, and do similarly for $H^k (\mathbb{R}^d)$ and $L^{r} (\mathbb{R}^d)$.

\item The pair $(q_0, r_0)$ denotes the admissible pair $\left( \frac{4(p+2)}{dp}, p+2\right)$.

\item We write `local'   to mean `local-in-time' for the sake of simplicity.

\item The notation `stable' means `stable uniformly in $\tau \in (0,1)$'.

\item $\nabla f$ always means $\nabla_x f$ even if $f$ is a time-space function.

\item In Section \ref{sec-10'}-\ref{sec-10}, we precisely write $Z_{\tau}^{\phi}$ and $u^{\phi}$ to denote the flow $Z_{\tau}$ and the solution $u$ corresponding to the initial data $\phi$.
\end{itemize}


\section{Preliminary lemmas}\label{sec-2}

In this section, we will introduce some basic lemmas that will be used throughout the paper.
Firstly, we state the Strichartz estimates for the Schr\"odinger equation.

\begin{thm}\label{thm-2-1}
Let $(q,r)$ and $(\tilde{q}, \tilde{r})$ be any admissible pairs. Then, there exist $C_{d,q}, C_{d,q,\widetilde{q}}>0$ such that
    \begin{equation}\label{eq-st-1}
    \|S_{\tau} (\cdot) \phi \|_{L^q (\mathbb{R}, L^r (\mathbb{R}^d))}
    \leq C_{d,q} \|\phi\|_{L^2 (\mathbb{R}^d)},
    \end{equation}
    $$
    \left\|\int_{\mathbb{R}} S_{\tau} (-s) f (s) ds \right\|_{L^2 (\mathbb{R}^d)}
    \leq C_{d,q} \|f\|_{L^{\tilde{q}'} ( \mathbb{R}, L^{\tilde{r}'}(\mathbb{R}^d))},
    $$
and
    $$
    \left\| \int_{s< t} S_{\tau} (t-s) f (s) ds \right\|_{L^q (\mathbb{R}, L^{r}(\mathbb{R}^d))}
    \leq C_{d,q,\tilde{q}} \|f\|_{L^{\tilde{q}'}(\mathbb{R}, L^{\tilde{r}'}(\mathbb{R}^d))}
    $$
hold for all $\phi\in L^2 (\mathbb{R}^d)$ and $f\in L^{\tilde{q}'}(\mathbb{R}, L^{\tilde{r}'}(\mathbb{R}^d))$.
\end{thm}

Strichartz \cite{St} proved \eqref{eq-st-1} for $S(t)$ with $q=r$. The other two estimates follow by a duality argument and the Christ-Kiselev lemma \cite{CK}.
Later, it was extended by Keel and Tao \cite{KT} to all admissible pairs, including the endpoint case $(q,r)=(2,\frac{2d}{d-2})$. These estimates for $S(t)$ are extended easily to the frequency localized operator $S_{\tau}(t)$ using that $S_{\tau}(t)\phi = S(t) (\Pi_\tau \phi)$ and \eqref{eq-2-20}.
Next we recall the time discrete versions of the Strichartz estimates, obtained by Ignat \cite{I}.
\begin{thm}[{\cite[Theorem 2.1]{I}}]\label{thm-str}
Let $(q,r)$ and $(\tilde{q}, \tilde{r})$ be any admissible pairs. Then, there exist $C_{d,q}, C_{d,q,\widetilde{q}}>0$ such that
    \begin{equation}\label{eq-st-4}
    \|S_{\tau}(\cdot) \phi \|_{\ell^q (\tau \mathbb{Z}; L^r (\mathbb{R}^d))}
    \leq C_{d,q}  \|\phi \|_{L^2 (\mathbb{R}^d)},
    \end{equation}
    $$
    \left\| \tau\sum_{n \in \mathbb{Z}} S_{\tau}(-n \tau) f (n \tau) \right\|_{L^2 (\mathbb{R}^d)}
    \leq C_{d,q} \|f\|_{\ell^{\tilde{q}'}(\tau \mathbb{Z}; L^{\tilde{r}'} (\mathbb{R}^d))} ,
    $$
and
    \begin{equation}\label{eq-st-6}
    \left\| \tau \sum_{k=-\infty}^{n-1}S_{\tau}((n-k) \tau) f (k\tau)\right\|_{\ell^q (\tau \mathbb{Z}; L^r (\mathbb{R}^d))}
    \leq C_{d,q,\tilde{q}} \|f\|_{\ell^{\tilde{q}'}(\tau \mathbb{Z}; L^{\tilde{r}'} (\mathbb{R}^d))}
    \end{equation}
hold for all $\phi\in L^2 (\mathbb{R}^d)$ and $f\in \ell^{\tilde{q}'}(\tau \mathbb{Z}; L^{\tilde{r}'}(\mathbb{R}^d))$.
\end{thm}

Since the operators $S_{\tau}$ and $\nabla$ are commutative, the above result immediately implies the following corollary.

\begin{cor}\label{thm-str2}
Let $(q,r)$ and $(\tilde{q}, \tilde{r})$ be any admissible pairs. Then, there exist $C_{d,q}, C_{d,q,\widetilde{q}}>0$ such that
    \begin{equation}\label{eq-st-7}
    \|S_{\tau}(\cdot) \phi \|_{\ell^q (\tau \mathbb{Z}; W^{1,r} (\mathbb{R}^d))}
    \leq C_{d,q}  \|\phi \|_{H^1 (\mathbb{R}^d)},
    \end{equation}
and
    \begin{equation}\label{eq-st-8}
    \left\| \tau \sum_{k=-\infty}^{n-1}S_{\tau}((n-k) \tau) f (k\tau)\right\|_{\ell^q (\tau \mathbb{Z}; W^{1,r} (\mathbb{R}^d))}
    \leq C_{d,q,\tilde{q}} \|f\|_{\ell^{\tilde{q}'}(\tau \mathbb{Z}; W^{1,\tilde{r}'} (\mathbb{R}^d))}
    \end{equation}
hold for all $\phi\in H^1 (\mathbb{R}^d)$ and $f\in \ell^{\tilde{q}'}(\tau \mathbb{Z}; W^{1,\tilde{r}'}(\mathbb{R}^d))$.
\end{cor}

By combining the Strichartz estimates in Theorems \ref{thm-2-1} and \ref{thm-str} with the Christ-Kiselev lemma \cite{CK}, the following result was derived.

\begin{cor}[{\cite[Lemma 4.5]{I}}]
For any admissible pairs $(q,r)$ and $(\tilde{q}, \tilde{r})$ with $(q,\tilde{q})\neq(2,2)$, we have
    \begin{equation}\label{eq-2-6}
    \left\| \int_{s < n \tau} S_{\tau}(n \tau -s) f(s) ds \right\|_{\ell^q (\tau \mathbb{Z}; L^r (\mathbb{R}^d))}
    \leq C_{d,q,\tilde{q}} \left\| f\right\|_{L^{\tilde{q}'} (\mathbb{R}; L^{\tilde{r}'} (\mathbb{R}^d))}.
    \end{equation}
\end{cor}

In the remaining part of this section, we prove some basic estimates that are used frequently in this paper.
\begin{lem}\label{lem-2-1}
For any $p \in (0, \infty)$, there exists a positive constant $c_p>0$ such that
    \begin{equation}\label{eq-2-7}
    \left| \frac{N(\tau) - I}{\tau} v - \frac{N(\tau) - I}{\tau}w \right|
    \leq c_p|v-w|(|v|^p + |w|^p)
    \end{equation}
and
    \begin{equation}\label{eq-2-10}
    \left| \frac{N(\tau)-I}{\tau} v\right|
    = \left| \frac{\exp(i \tau \lambda|v|^p) -1}{\tau} v\right| \leq |v|^{p+1}
    \end{equation}
hold for all $v, w \in \mathbb{C}$.
Furthermore, for weakly differentiable $f : \mathbb{R}^d \rightarrow \mathbb{C}$, we have a pointwise estimate
    \begin{equation}\label{eq-2-13}
    \left| \nabla \left( \frac{N(\tau) - I}{\tau} f \right) \right| \leq (p+1) |f|^p |\nabla f|.
    \end{equation}
\end{lem}

\begin{proof}
The estimates \eqref{eq-2-7} and \eqref{eq-2-10} follow from a direct calculation using the mean value theorem (see also Lemma 4.2 in \cite{I}).
For the last estimate, we notice that
    $$
    \begin{aligned}
    \left| \nabla \left( \frac{N(\tau) - I}{\tau} f \right) \right|
    &= \left| \nabla \left( \frac{\exp(i\tau \lambda |f|^p) - 1}{\tau} f \right) \right| \\
    &\leq \left|\left( \frac{\exp(i\tau \lambda |f|^p) - 1}{\tau} \right)\nabla f \right|
    + \Big|\exp (i\tau  \lambda|f|^p)~ p|f|^p \nabla f \Big| \\
    &\leq (p+1)|f|^{p} |\nabla f|,
    \end{aligned}
    $$
where \eqref{eq-2-10} is used for the last inequality. The lemma is proved.
\end{proof}

\begin{lem}\label{lem-2-6}
For any $1\leq q \leq r <\infty$ and $\phi:\mathbb{R}^d\rightarrow\mathbb{C}$, we have
    \begin{equation}\label{eq-2-20}
    \big\| \Pi_{\tau}\phi - \phi \big\|_{L^{r} (\mathbb{R}^d)}
    \leq C\tau^{1/2} \big\| (-\Delta)^{1/2} \phi \big\|_{L^{r} (\mathbb{R}^d)} ,
    \end{equation}
    \begin{equation}\label{eq-2-20'}
    \| \Pi_{\tau} \phi\|_{L^r (\mathbb{R}^d)}
    \leq C \|\phi\|_{L^r (\mathbb{R}^d)} ,
    \end{equation}
    \begin{equation}\label{eq-2-21}
    \big\| \nabla (\Pi_{\tau} \phi) \big\|_{L^r (\mathbb{R}^d)}
    \leq C \tau^{-\frac{1}{2}} \| \phi \|_{L^r (\mathbb{R}^d)},
    \end{equation}
and
    \begin{equation}\label{eq-2-22}
    \|\Pi_{\tau} \phi \|_{L^{r}(\mathbb{R}^d)}
    \leq C \tau^{\frac{d}{2} \left( \frac{1}{r}-\frac{1}{q}\right)} \| \phi\|_{L^{q}(\mathbb{R}^d)}.
    \end{equation}
\end{lem}

\begin{proof}
The estimate \eqref{eq-2-20} and \eqref{eq-2-20'} follow from the basic multiplier theory
(see, e.g., Theorem 4.4 in \cite{I} and Theorem 5.2.2 in \cite{G}).
In order to show that \eqref{eq-2-21} and \eqref{eq-2-22},
we notice from the definition \eqref{eq-1-10} that
    \begin{equation}\label{eq-2-23}
    \Pi_{\tau} \phi(x) = (K_{\tau} * \phi ) (x),
    \end{equation}
where $K_{\tau} (x) = \tau^{-\frac{d}{2}} \widehat{\chi} (\tau^{-1/2} x)$.
Since $\chi \in C^N (B^d(0,2))$ with $N=2d$, its Fourier transform $\widehat{\chi}$ admits the decay property $|\widehat{\chi}(\xi)| \leq C_{d,N} (1+|\xi|)^{-N}$. In addition, by the equality $\partial_{\xi_i} \widehat{\chi}(\xi) = \widehat{x_i \chi} (\xi)$ we have $|\partial_{\xi_i} \hat{\chi}(\xi)| \leq C_{d,N} (1+|\xi|)^{-N}$, and
    \begin{equation}\label{eq-2-24}
    \partial_{x_i} \big( \Pi_{\tau} \phi \big) (x)
    = \tau^{-\frac{d+1}{2}} \big( \partial_{x_i} \widehat{\chi} \big)(\tau^{-1/2} \cdot)* \phi (x).
    \end{equation}
By applying Young's inequality, we obtain that
    $$
    \begin{aligned}
    \big\| \nabla \big( \Pi_{\tau} \phi \big) \big\|_{L^r (\mathbb{R}^d)}
    &\leq C \tau^{-\frac{d+1}{2}} \big\| \nabla \widehat{\chi} (\tau^{-1/2} \cdot) \big\|_{L^1 (\mathbb{R}^d)} \| \phi\|_{L^r (\mathbb{R}^d)} \\
    &\leq C \tau^{-\frac{1}{2}} \| \phi\|_{L^r (\mathbb{R}^d)},
    \end{aligned}
    $$
which gives estimate \eqref{eq-2-21}.
By applying Young's inequality to \eqref{eq-2-23} with $q<r$,
	\begin{equation*}
	\bigl\|\Pi_{\tau} \phi\bigr\|_{L^{r} (\mathbb{R}^d)} \leq C \big\|K_{\tau}\big\|_{L^{\alpha}(\mathbb{R}^d)} \|\phi\|_{L^{q}(\mathbb{R}^d)},
	\end{equation*}
where $1/r + 1 = 1/\alpha + 1/q$. This verifies the estimate \eqref{eq-2-22}. The proof is finished.
\end{proof}

\begin{lem}\label{lem-5-2}
For any admissible pairs $(q,r)$ and $(\widetilde{q},\widetilde{r})$, there is a constant $C_{d,q,\widetilde{q}}>0$ such that
    \begin{equation}\label{eq-5-5}
    \begin{split}
    &\left\| \int_{s< n \tau} S_{\tau}(n \tau- s) f(s) ds
        - \tau \sum_{k=-\infty}^{n-1} S_{\tau}(n \tau - k \tau) f(k \tau) \right\|_{\ell^q (\tau \mathbb{Z}, L^r (\mathbb{R}^d))} \\
    &\qquad\leq C_{d,q,\widetilde{q}}~ \tau^{1/2} \|f\|_{L^{\tilde{q}'}(\mathbb{R}, W^{1,\tilde{r}'}(\mathbb{R}^d))}
        + C_{d,q,\widetilde{q}}~ \tau \| \partial_t f \|_{L^{\tilde{q}'}(\mathbb{R}, L^{\tilde{r}'}(\mathbb{R}^d))}
    \end{split}
    \end{equation}
hold for all test function $f \in \mathcal{S}(\mathbb{R}^{d+1})$.
\end{lem}

\begin{proof}
First, we recall the following estimate from Lemma 4.3 and Lemma 4.6 in \cite{I}:
    \begin{equation}\label{eq-5-1}
    \begin{split}
    &\left\| \int_{s< n \tau} S_{\tau}(n \tau- s) f(s) ds
        - \tau \sum_{k=-\infty}^{n-1} S_{\tau}(n \tau - k \tau) f(k \tau) \right\|_{\ell^q (\tau \mathbb{Z}, L^r (\mathbb{R}^d))} \\
    &\qquad\leq C_{d,q,\widetilde{q}}~ \tau \| \nabla^2 f \|_{L^{\tilde{q}'}(\mathbb{R}, L^{\tilde{r}'}(\mathbb{R}^d))}
        + C_{d,q,\widetilde{q}}~ \tau \| \partial_t f \|_{L^{\tilde{q}'}(\mathbb{R}, L^{\tilde{r}'}(\mathbb{R}^d))},
    \end{split}
    \end{equation}
where $(q,r)$ and $(\widetilde{q},\widetilde{r})$ are any admissible pairs.
We notice that $S_{\tau} (t) f(x) = S_{\tau}(t) \Pi_{\tau/4} f(x)$, by definition of $S_{\tau}$. Using this and \eqref{eq-5-1} with Lemma \ref{lem-2-6}, we obtain that
    $$
    \begin{aligned}
    &\left\| \int_{s< n \tau} S_{\tau}(n \tau- s) f(s) ds
        - \tau \sum_{k=-\infty}^{n-1} S_{\tau}(n \tau - k \tau) f(k \tau) \right\|_{\ell^q (\tau \mathbb{Z}, L^r (\mathbb{R}^d))} \\
    &\qquad\leq C_{d,q,\widetilde{q}}~ \tau \left\| \nabla^2 \big(\Pi_{\tau/4} f \big)\right\|_{L^{\tilde{q}'}(\mathbb{R}, L^{\tilde{r}'}(\mathbb{R}^d))}
        + C_{d,q,\widetilde{q}}~ \tau \left\| \partial_t \big(\Pi_{\tau/4} f \big) \right\|_{L^{\tilde{q}'}(\mathbb{R}, L^{\tilde{r}'}(\mathbb{R}^d))} \\
    &\qquad\leq C_{d,q,\widetilde{q}}~ \tau^{1/2} \left\|  \nabla f \right\|_{L^{\tilde{q}'}(\mathbb{R}, L^{\tilde{r}'}(\mathbb{R}^d))}
        + C_{d,q,\widetilde{q}}~ \tau \left\| \partial_t f \right\|_{L^{\tilde{q}'}(\mathbb{R}, L^{\tilde{r}'}(\mathbb{R}^d))}.
    \end{aligned}
    $$
This proves the estimate \eqref{eq-5-5}.
\end{proof}

\begin{lem}\label{LemSov_2p+1}
Under the assumption \eqref{def_C_1}, we have
    \begin{equation}\label{Sov_2p+1}
    \big\| |\Pi_{\tau} u|^{2p+1} \big\|_{L^{q_0'} (0,T; L^{r_0'})}
    + \Big\| |\Pi_{\tau} u|^p \Pi_{\tau} (|u|^p u) \Big\|_{L^{q_0'} (0,T; L^{r_0'})}
    \leq C_{d,p}~ \tau^{-1/2} \big( 1 + T^{\frac{1}{q_0'}} \big) M_1^{2p+1} .
    \end{equation}
\end{lem}

\begin{proof}
The proof is a simple combination of H\"older's inequality and Sobolev's embedding.
We consider the first term of left hand side in \eqref{Sov_2p+1}.
By Lemma \ref{lem-2-6}, we have
    \begin{equation}\label{eq-(2p+1)}
    \big\| |\Pi_{\tau} u|^{2p+1} \big\|_{L^{q_0'} (0, T; L^{r_0'})} \\
    \leq C_{d,p} \tau^{-\frac{d}{2} (\frac{2p+1}{r_1}-\frac{1}{r_0'})} \| u \|_{L^{(2p+1)q_0'} (0,T; L^{r_1})}^{2p+1}
    \end{equation}
for all $r_1 \leq (2p+1)r_0'$.
We choose the value of $r_1 >0$ separately for the cases that $\frac{3}{d} \leq p < p_d$ and $0< p < \frac{3}{d}$ such as
    $$
    \frac{1}{r_1} =
    \biggl\{\begin{array}{ll} \frac{1}{2p+1} \big( \frac{p+1}{p+2} +\frac{1}{d} \big) &\quad \textrm{if}\quad \frac{3}{d} \leq p < p_d , \\
    \frac{1}{(2p+1)r_0'} &\quad \textrm{if}\quad 0< p < \frac{3}{d}.
    \end{array}\biggr.
    $$
First, we consider the case that $\frac{3}{d} \leq p < p_d$. In this case, we take $\frac{1}{r_1} =\frac{1}{2p+1} \big( \frac{p+1}{p+2} +\frac{1}{d} \big) <1$ and choose $q_2$ and $r_2$ such that
    $$
    \frac{1}{q_2} := \frac{1}{(2p+1)q_0'} = \frac{4(p+2)-dp}{4(p+2)(2p+1)}
    \quad\mbox{and}\quad
    \frac{1}{r_2} := \frac{1}{2} - \frac{2}{dq_2} = \frac{1}{2} - \frac{4(p+2)-dp}{2d(p+2)(2p+1)}
    $$
so that $(q_2, r_2)$ is an admissible pair whenever $\frac{3}{d}\leq p < p_d$.
Then, by \eqref{eq-(2p+1)} and the Sobolev embedding $W^{s,r_2}\subseteq L^{r_1}$ (with $\frac{1}{r_1}+\frac{s}{d}=\frac{1}{r_2}$), we obtain that
    $$
    \begin{aligned}
    \big\| |\Pi_{\tau} u|^{2p+1} \big\|_{L^{q_0'} (0,T; L^{r_0'})}
    &\leq C_{d,p}\tau^{-\frac{1}{2}} \| u \|_{L^{q_2} (0,T; L^{r_1})}^{2p+1} \\
    &\leq C_{d,p} \tau^{-\frac{1}{2}} \| u \|_{L^{q_2} (0,T; W^{s,r_2})}^{2p+1} ,
    \end{aligned}
    $$
where
    $$
    s = \frac{d}{2} - \frac{d+6}{2(2p+1)} .
    $$
One may check that $s \in [0,1]$ whenever $\frac{3}{d} \leq p \leq p_d$. Therefore,
we have that
    $$
    \begin{aligned}
    \big\| |\Pi_{\tau} u|^{2p+1} \big\|_{L^{q_0'} (0,T; L^{r_0'})}
    &\leq C_{d,p} \tau^{-\frac{1}{2}} \| u \|_{L^{q_2} (0,T; W^{s,r_2})}^{2p+1} \\
    &\leq C_{d,p} \tau^{-\frac{1}{2}} \| u \|_{L^{q_2} (0,T; W^{1,r_2})}^{2p+1}
    \leq C_{d,p} \tau^{-\frac{1}{2}} M_1^{2p+1}
    \end{aligned}
    $$
from \eqref{def_C_1}.
For the case that $0<p <\frac{3}{d}$, we set $r_1 = (2p+1)r_0'$ in \eqref{eq-(2p+1)}. Using H\"older's inequality in $t$-variable and the Sobolev embedding, we have that
    $$
    \begin{aligned}
    \big\| |\Pi_{\tau} u|^{2p+1} \big\|_{L^{q_0'} (0,T; L^{r_0'})}
    \leq T^{\frac{1}{q_0'}} \| u \|_{L^{\infty} (0,T; L^{(2p+1)r_0'})}^{2p+1}
    \leq C_{d,p} T^{\frac{1}{q_0'}} \| u \|_{L^{\infty} (0,T; H^{s})}^{2p+1} ,
    \end{aligned}
    $$
where
    $$
    s = \frac{d}{2} - \frac{d(p+1)}{(p+2)(2p+1)} .
    $$
Because $s \in [0,1]$ whenever $0\leq p \leq \frac{3}{d}$,
this is bounded by $C_{d,p} T^{\frac{1}{q_0'}} M_1^{2p+1}$ thanks to \eqref{def_C_1}.
Thus, we have the upper bound
    $$
    \big\| |\Pi_{\tau} u|^{2p+1} \big\|_{L^{q_0'} (0,T; L^{r_0'})}
    \leq C_{d,p} \big( \tau^{-\frac{1}{2}} + T^{\frac{1}{q_0'}} \big) M_1^{2p+1}
    $$
for all $0<p < p_d$.

Next, for the second term of left hand side in \eqref{Sov_2p+1},
we apply H\"older's inequality and Lemma \ref{lem-2-6} again to obtain
    $$
    \begin{aligned}
    \Big\| |\Pi_{\tau} u|^p \Pi_{\tau} (|u|^p u) \Big\|_{L^{q_0'} (0,T; L^{r_0'})}
    &\leq \big\| \Pi_{\tau} u \big\|_{L^{(2p+1)q_0'} (0,T; L^{(2p+1)r_0'})}^p
        \Big\| \Pi_{\tau} (|u|^p u) \Big\|_{L^{\frac{2p+1}{p+1}q_0'} (0,T; L^{\frac{2p+1}{p+1}r_0'})} \\
    &\leq C\tau^{-\frac{d}{2} (\frac{2p+1}{r_1}-\frac{1}{r_0'})} \| u \|_{L^{(2p+1)q_0'} (0,T; L^{r_1})}^p
        \big\| |u|^p u \big\|_{L^{\frac{2p+1}{p+1}q_0'} (0,T; L^{\frac{r_1}{p+1}})} \\
    &= C\tau^{-\frac{d}{2} (\frac{2p+1}{r_1}-\frac{1}{r_0'})} \| u \|_{L^{(2p+1)q_0'} (0,T; L^{r_1})}^{2p+1}
    \end{aligned}
    $$
for any $r_1 \leq (2p+1)r_0'$. The right hand side is same with \eqref{eq-(2p+1)}. Hence the desired estimate follows in the same way.
\end{proof}

In the last of section, we recall H\"older's inequality for $L^{q_0'}(I; L^{r_0'}(\mathbb{R}^d))$ which we use frequently:
    \begin{equation}\label{freq_Hol}
    \big\| |f|^{p}g \|_{L^{q_0'}(I; L^{r_0'})}
    \leq |I|^{\frac{1}{q_0'}-\frac{1}{q_0}} \| f \|_{L^{\infty} (I; L^{r_0})}^p \| g \|_{L^{q_0}(I; L^{r_0})}.
    \end{equation}


\section{The $L^2$ convergence result between $Z_\tau$ with $u$}\label{sec-6}

In this section, we prove Theorem \ref{thm-rel}. The key ingredient is to use the Duhamel formulas of $Z_{\tau}$ and $u$ given in \eqref{eq-1-9} and \eqref{eq-1-8} respectively.

\begin{proof}[Proof of Theorem \ref{thm-rel}]
Let us fix a time $T \in (0,\infty)$ such that $\sup_{0\leq t \leq T}\|u(t)\|_{H^1} <\infty$.
Then, from Theorem \ref{wp_u_H12} and the assumption of Theorem \ref{thm-rel}, we see that $u$ and $Z_{\tau}$ satisfy the following estimates
    $$
	\begin{aligned}
	\|u\|_{L^\infty (0,T; H^1)} + \|u\|_{L^{q_0} (0,T; W^{1,r_0})}
	&\leq M_1, \\
	\| Z_{\tau}(n \tau)\|_{\ell^{\infty} (0,T; H^1)} +
	\| Z_{\tau}(n \tau)\|_{\ell^{q_0} (0,T; W^{1,r_0} )}
	&\leq M_2.
	\end{aligned}
    $$
For our purpose, it is sufficient to estimate $Z_{\tau}(n \tau) - \Pi_{\tau} u(n \tau)$ instead of $Z_{\tau}(n \tau) - u(n \tau)$, because we have that
    $$
    \big\| u(n \tau) - \Pi_{\tau} u(n \tau) \big\|_{\ell^{\infty} (0,T; L^2)}
    \leq C\tau^{1/2} \| u(n \tau)\|_{\ell^{\infty} (0,T; H^1)}
    \leq C\tau^{1/2} M_1
    $$
by \eqref{eq-2-20} and \eqref{def_C_1}.
Now, we take $T_*>0$ as
    \begin{equation}\label{eq-4-16}
    T_* = \alpha_{d,p} \big( M_1 + M_2 \big)^{-\frac{2p(p+2)}{4-(d-2)p}}
    \end{equation}
where a small constant $\alpha_{d,p}>0$ will be chosen later.

First, for the case of $\tau \in [T_* /4,1)$, we can obtain the desired estimate easily as below:
    \begin{equation*}
    \begin{split}
    \max_{\tau \in \mathbb{N}} \|Z_{\tau}(n\tau) - u(n\tau)\|_{L^2}
    &\leq \max_{\tau \in \mathbb{N}} \|Z_{\tau}(n\tau) \|_{L^2}  +\max_{\tau \in \mathbb{N}} \|u(n\tau)\|_{L^2} \\
    &\leq 2\|\phi\|_{L^2} \leq \tau^{\frac{1}{2}} \Big(\frac{4}{\sqrt{T_*}}\Big) \|\phi\|_{L^2} \\
    &\leq \tau^{\frac{1}{2}} C_{d,q} \big( M_1 + M_2 \big)^{\frac{p(p+2)}{4-(d-2)p}}  \|\phi\|_{L^2}.
    \end{split}
    \end{equation*}

Now, assume $0< \tau < T_* /4$. We take $R \in (T_{*}/2, T_{*}]$ such that $R/\tau \in \mathbb{N}$. To proceed an induction, we split $[0,T]$ as
\begin{equation}
\begin{split}
[0,T]&= \cup_{j=0}^{N-1} [jR, (j+1)R) \, \cup \, [NR, T)
\\
&=: \cup_{j=0}^{N-1} I_j \, \cup \, I_N,
\end{split}
\end{equation}
where $N \in \mathbb{N}$ is chosen so that $NR \leq T< (N+1) R$. For each $j \in \{0,1, \cdots, N\}$ we choose $m_j \in \mathbb{N}$ such that $m_j \tau = jR$, i.e., $m_j = j (R/\tau)$. Then
    \begin{equation}\label{eq-5-57}
    \big\| Z_{\tau} - \Pi_{\tau} u \big\|_{\ell^{q} (I_{j}; L^{r})}= \Bigl\|Z_{\tau}(m_j \tau +n \tau) - \Pi_{\tau} u(m_j \tau +n \tau) \Bigr\|_{\ell^{q}(0, R; L^r)},
    \end{equation}
    where, if $j=N$, we regard the interval $(0,R)$ is replaced by $(0, T-NR) \subset (0,T)$. By considering $Z_{\tau}(m_j \tau)$ as initial data for each $j = 0,1,\cdots, N$, the formula \eqref{eq-1-9} can be written as
    \begin{equation}\label{eq-a-1'}
    Z_{\tau}(m_j\tau +n\tau)
    =S_{\tau}(n\tau) Z_{\tau}(m_j\tau) +
    \tau \sum_{k=0}^{n-1} S_{\tau}(n \tau -k \tau) \frac{N (\tau) - I}{\tau} Z_{\tau}(m_j\tau + k \tau), \quad n \geq 1.
    \end{equation}
By combining this with \eqref{eq-1-8}, we obtain the following decomposition:
    $$
    Z_{\tau}(m_j \tau +n \tau) - \Pi_{\tau} u(m_j \tau +n \tau)
    = \mathcal{A}_1(j) + \mathcal{A}_2(j) + \mathcal{A}_3(j) + \mathcal{A}_4(j) ,
    $$
where
    $$
    \begin{aligned}
    &\mathcal{A}_1(j) := S_{\tau} (n\tau)\Bigl( Z_{\tau}(m_j \tau) - \Pi_{\tau} u(m_j \tau)\Bigr)
    \\
    &\mathcal{A}_2(j) := S_{\tau}(n\tau) \Bigl( \Pi_{\tau} u (m_j \tau)-u(m_j \tau)\Bigr)
    \\
    &\mathcal{A}_3(j) := \tau \sum_{k=0}^{n-1} S_{\tau}(n \tau -k \tau) \left( \frac{N(\tau) - I}{\tau} Z_{\tau}(m_j\tau+k\tau)
        - \frac{N(\tau) - I}{\tau} \Pi_{\tau} u(m_j\tau+k\tau) \right)
    \\
    &\mathcal{A}_4(j) := \tau \sum_{k=0}^{n-1} S_{\tau}(n \tau -k \tau) \frac{N(\tau) - I}{\tau} \Pi_{\tau} u(m_j\tau+k\tau)
        - i\lambda\int_0^{n\tau} S_{\tau}(n\tau -s) |u|^p u (m_j\tau+s) ds.
    \end{aligned}
    $$
For $(q,r) \in \{(q_0, r_0), (\infty,2)\}$, we apply the triangle inequality to obtain that
    \begin{equation}\label{eq-5-40}
    \begin{split}
    &\Big\| Z_{\tau}(m_j \tau +n \tau) - \Pi_{\tau} u(m_j \tau +n \tau) \Big\|_{\ell^{q}(0, R; L^r)}
    \\
    &\quad\leq \| \mathcal{A}_1(j) \|_{\ell^{q}(0, R; L^r)} + \| \mathcal{A}_2(j) \|_{\ell^{q}(0, R; L^r)} + \| \mathcal{A}_3(j) \|_{\ell^{q}(0,R; L^r)} + \| \mathcal{A}_4(j) \|_{\ell^{q}(0,R; L^r)}.
    \end{split}
    \end{equation}
To estimate right hand side of \eqref{eq-5-40}, by using the Strichartz estimate \eqref{eq-st-4}, we get the estimate
    \begin{equation}\label{eq-5-42}
    \begin{split}
    \| \mathcal{A}_1(j) \|_{\ell^{q}(0, R; L^r)}
    &= \Bigl\|S_{\tau} (n\tau)\Bigl( Z_{\tau}(m_j \tau) - \Pi_{\tau} u(m_j \tau)\Bigr)\Bigr\|_{\ell^{q}(0,R; L^r)} \\
    &\leq C_{d,q} \Big\| Z_{\tau}(m_j \tau) - \Pi_{\tau} u (m_j \tau) \Big\|_{L^2}.
    \end{split}
    \end{equation}
Secondly, we use the inequality \eqref{eq-st-4}, \eqref{eq-2-20} with \eqref{def_C_1} to deduce that
    \begin{equation}\label{eq-5-42'}
    \begin{split}
    \| \mathcal{A}_2(j) \|_{\ell^{q}(0,R; L^r)}
    &= \Bigl\|S_{\tau}(n\tau) \Bigl( \Pi_{\tau} u(m_j \tau) - u (m_j \tau)\Bigr)\Bigr\|_{\ell^q (0,R; L^r)} \\
    &\leq C_{d,q} \big\| \Pi_{\tau} u(m_j \tau) - u (m_j \tau) \big\|_{L^2} \\
    &\leq C_{d,q} \tau^{1/2} \big\| (-\Delta)^{1/2} u(m_j \tau) \big\|_{L^2}
    \leq C_{d,q} \tau^{1/2} M_1.
    \end{split}
    \end{equation}
Next, to estimate $\mathcal{A}_3(j)$, we use \eqref{eq-2-7} to find that
    \begin{equation}\label{eq-6-5}
    \left| \frac{N(\tau) - I}{\tau} Z_{\tau} - \frac{N(\tau) - I}{\tau} \Pi_{\tau} u \right|
    \leq c_p |Z_{\tau} - \Pi_{\tau} u| \Big( |Z_{\tau}|^p +|\Pi_{\tau}u|^p \Big).
    \end{equation}
After applying the Strichartz estimate \eqref{eq-st-6} to $\mathcal{A}_3(j)$, we use \eqref{eq-6-5} with H\"older's inequality \eqref{freq_Hol}  and the fact that $R\leq T_*$ to find
    \begin{equation}\label{eq-6-2}
    \begin{split}
    \| \mathcal{A}_3(j) \|_{\ell^{q}(0, R; L^r)}
    &\leq C_{d,p}  \left\| |Z_{\tau} - \Pi_{\tau} u| \Big( |Z_{\tau}|^p +|\Pi_{\tau}u|^p \Big) (n\tau + m_j \tau)\right\|_{\ell^{q_0'} (0,R; L^{r_0'} )} \\
    &\leq C_{d,p} T_*^{\frac{1}{q_0'} -\frac{1}{q_0}} \big\| (Z_{\tau} - \Pi_{\tau} u)(n\tau + m_j \tau) \big\|_{\ell^{q_0} (0,R; L^{r_0})}
    \\
    &\quad \times \Big(  \big\| Z_{\tau}(n\tau + m_j \tau) \big\|_{\ell^{\infty} (0,R; L^{r_0})}^p + \big\| \Pi_{\tau}u (n\tau + m_j \tau) \big\|_{\ell^{\infty} (0,R; L^{r_0})}^p \Big).
    \end{split}
    \end{equation}
To proceed further, we use the Sobolev embedding $H^1 (\mathbb{R}^d) \rightarrow L^{r_0} (\mathbb{R}^d)$ (thanks to $r_0 <p_d +2$), \eqref{def_C_1} and \eqref{eq-5-51} to obtain
    $$
    \begin{aligned}
    \big\| \Pi_{\tau}u(n\tau + m_j \tau) \big\|_{\ell^{\infty} (0,R; L^{r_0})}
    \leq C_{d,p} \| u \|_{\ell^{\infty} (0,R; H^1)}
    &\leq C_{d,p} M_1, \\
    \big\| Z_{\tau}(n\tau + m_j \tau) \big\|_{\ell^{\infty} (0,R; L^{r_0})}
    \leq  C_{d,p} \| Z_{\tau} \|_{\ell^{\infty} (0,R; H^1)}
    &\leq C_{d,p} M_2.
    \end{aligned}
    $$
By inserting these estimates into \eqref{eq-6-2}, we get
    \begin{equation}\label{eq-r-34}
    \| \mathcal{A}_3(j) \|_{\ell^{q}(0, R; L^r)}
    \leq C_{d,p} ( M_1^p + M_2^p ) T_*^{\frac{4-(d-2)p}{2(p+2)}}
    \big\| (Z_{\tau} - \Pi_{\tau} u)(n\tau + m_j \tau) \big\|_{\ell^{q_0} (0,R; L^{r_0})},
    \end{equation}
where we  used the identity
    \begin{equation}\label{eq-r-54}
    \frac{1}{q_0'} - \frac{1}{q_0} = \frac{4-(d-2)p}{2(p+2)}.
    \end{equation}
Now, by choosing $\alpha_{d,p} >0$ small enough in \eqref{eq-4-16}, we deduce from \eqref{eq-r-34} the following estimate
    \begin{equation}\label{eq-5-53'}
    \| \mathcal{A}_3(j) \|_{\ell^{q}(0, R; L^r)}
    \leq \frac{1}{2} \big\| (Z_{\tau} - \Pi_{\tau} u)(n\tau + m_j \tau) \big\|_{\ell^{q_0} (0,R; L^{r_0})}.
    \end{equation}
Lastly, for $\mathcal{A}_4(j)$, we claim the following estimate:
    \begin{equation}\label{eq-5-53}
    \max_{j =0,1,\cdots,N} \Big\{
    \| \mathcal{A}_4(j) \|_{\ell^{\infty}(0, R; L^2)}
    + \| \mathcal{A}_4(j) \|_{\ell^{q_0}(0, R; L^{r_0})} \Big\}
    \leq C_{d,p}~ \tau^{1/2} \big( 1+ T_*^{\frac{1}{q_0'}} \big) \big( 1+ M_1^{2p+1} \big),
    \end{equation}
whose verification is given in Lemma \ref{lem-K1} below. Now we are ready to finish the proof. From \eqref{eq-5-42}, \eqref{eq-5-42'}, \eqref{eq-5-53'} and \eqref{eq-5-53}, we have
    \begin{equation}\label{eq-r-60}
    \begin{split}
    &\Bigl\|Z_{\tau}(m_j \tau +n \tau) - \Pi_{\tau} u(m_j \tau +n \tau) \Bigr\|_{\ell^{q}(0,R; L^r)}
    \\
    &\quad\leq C_{d,q} \big\| Z_{\tau}(m_j \tau) - \Pi_{\tau} u (m_j \tau)\big\|_{L^2} +\frac{1}{2}\Bigl\|Z_{\tau}(m_j \tau +n \tau) - \Pi_{\tau} u(m_j \tau +n \tau) \Bigr\|_{\ell^{q_0}(0,R; L^{r_0})} \\
    &\qquad + C_{d,p}~ \tau^{1/2} \big( 1+ T_*^{\frac{1}{q_0'}} \big) \big( 1+ M_1^{2p+1} \big).
    \end{split}
    \end{equation}
Since \eqref{eq-r-60} holds for $(q,r) \in \{(q_0, r_0), (\infty,2)\}$,
we get
    $$
    \begin{aligned}
    &\Bigl\|Z_{\tau}(m_j \tau +n \tau) - \Pi_{\tau} u(m_j \tau +n \tau) \Bigr\|_{\ell^{\infty}(0, R; L^{2})} \\
    &\qquad\leq  C_{d,p}\big\| Z_{\tau}(m_j \tau) - \Pi_{\tau} u (m_j \tau)\big\|_{L^2}
    + C_{d,p}~ \tau^{1/2} \big( 1+ T_*^{\frac{1}{q_0'}} \big) \big( 1+ M_1^{2p+1} \big).
    \end{aligned}
    $$
Then, by utilizing \eqref{eq-5-57}, we arrive at the following estimates:
    $$
    \Biggl\{\begin{aligned}
    &\big\| Z_{\tau} - \Pi_{\tau} u \big\|_{\ell^{\infty} (I_{0}; L^{2})}
    \leq C_{d,p}~ \tau^{1/2} \big( 1+ T_*^{\frac{1}{q_0'}} \big) \big( 1+ M_1^{2p+1} \big) ,
    \\
    &\big\| Z_{\tau} - \Pi_{\tau} u \big\|_{\ell^{\infty} (I_{j+1}; L^{2})}
    \leq C_{d,p}\big\| Z_{\tau} - \Pi_{\tau} u \big\|_{\ell^{\infty} (I_{j}; L^{2})}
    + C_{d,p}~ \tau^{1/2} \big( 1+ T_*^{\frac{1}{q_0'}} \big) \big( 1+ M_1^{2p+1} \big),
    \end{aligned}\Biggr.
    $$
for $j=0,1, \cdots, N-1$.
Inductively, this implies that
    $$
    \big\| Z_{\tau} - \Pi_{\tau} u \big\|_{\ell^{\infty} (I_{j}; L^2)}
    \leq (C_{d,p})^{j+1}~ \tau^{1/2} \big( 1+ T_*^{\frac{1}{q_0'}} \big) \big( 1+ M_1^{2p+1} \big)
    $$
for all $j=0,1,\cdots,N$.
Since $N \leq 2T/T_*$ with $T_*$ given in \eqref{eq-4-16}, we get
        $$
    \begin{aligned}
    \big\| Z_{\tau} - \Pi_{\tau} u \big\|_{\ell^{\infty} (0,T; L^2)}
    &\leq \tau^{1/2} \big( M_1 + M_2 \big)^{\frac{2p(p+2)}{4-(d-2)p}}  \sum_{j=0}^{[2T/T^*]} (C_{d,p})^{j+1}
    \\
    &\leq \tau^{1/2} \exp\Big( \widetilde{C}_{d,p} T \big( M_1 + M_2 \big)^{\frac{2p(p+2)}{4-(d-2)p}} \Big)
    \end{aligned}
    $$
with a suitable constant $\widetilde{C}_{d,p}>0$, which completes the proof of the theorem.
\end{proof}

In the remaining part of this section,  we prove the estimate \eqref{eq-5-53} used in the above proof.


\begin{lem}\label{lem-K1}
Under the assumption \eqref{def_C_1}, we have
    \begin{equation}\label{eq-5-10K}
    \begin{split}
    &\left\|  \tau \sum_{k=0}^{n-1} S_{\tau}(n \tau -k \tau) \frac{N(\tau)-I}{\tau} \Pi_{\tau} u (k \tau)
        - i\lambda\int_0^{n\tau} S_{\tau}(n\tau -s) |u|^p u (s) ds\right\|_{\ell^q (0, T; L^r)} \\
    &\qquad\leq C_{d,p}~ \tau^{1/2} \big( 1+ T^{\frac{1}{q_0'}} \big) \big( 1+ M_1^{2p+1} \big).
    \end{split}
    \end{equation}
for all admissible pairs $(q,r)$ and $\tau \in (0,1)$.
Here, the constant $M_1 > 0$ is referred to \eqref{def_C_1}.
\end{lem}

\begin{proof}
Recall that $(q_0,r_0)$ denote the admissible pair $(\frac{4(p+2)}{dp},p+2)$.
Also, we denote that
    $$
    \begin{aligned}
    &\mathcal{B}_1(u) := \frac{N(\tau) -I}{\tau} \Pi_{\tau} u (s)
    \\
    &\mathcal{B}_2(u) := \frac{N(\tau) - I}{\tau} \Pi_{\tau} u -i\lambda |\Pi_{\tau}u|^p \Pi_{\tau}u
    \\
    &\mathcal{B}_3(u) := i\lambda\Big( |\Pi_{\tau}u|^p \Pi_{\tau}u - |u|^p u \Big)
    \end{aligned}
    $$
Now, in order to show \eqref{eq-5-10K}, we perform the decomposition
    \begin{equation}\label{eq-lem-K1}
    \begin{aligned}
    &\tau \sum_{k=0}^{n-1} S_{\tau}(n \tau -k \tau) \frac{N(\tau)-I}{\tau} \Pi_{\tau} u (k \tau)
        - i\lambda\int_0^{n\tau} S_{\tau}(n\tau -s) |u|^p u (s) \,ds \\
    &\quad = \bigg( \tau \sum_{k=0}^{n-1} S_{\tau}(n \tau -k \tau) \frac{N(\tau)-I}{\tau} \Pi_{\tau} u (k \tau)
        - \int_0^{n\tau} S_{\tau}(n\tau -s) \frac{N(\tau) -I}{\tau} \Pi_{\tau} u (s) \,ds \bigg)\\
    &\qquad\quad + \bigg( \int_0^{n\tau} S_{\tau}(n\tau -s) \frac{N(\tau) -I}{\tau} \Pi_{\tau} u (s)\,ds
        - i\lambda\int_0^{n\tau} S_{\tau}(n\tau -s) |u|^p u (s)\,ds \bigg).
    \end{aligned}
    \end{equation}
By Lemma \ref{lem-5-2}, we can find an upper bound of the first term of \eqref{eq-lem-K1} as
    $$
    \begin{aligned}
    &\left\| \tau \sum_{k=0}^{n-1} S_{\tau}(n \tau -k \tau) \frac{N(\tau)-I}{\tau} \Pi_{\tau} u (k \tau)
        - \int_0^{n\tau} S_{\tau}(n\tau -s) \frac{N(\tau) -I}{\tau} \Pi_{\tau} u (s) ds \right\|_{\ell^q (0, T; L^r)} \\
    &\qquad\leq C\,\tau^{1/2} \left\| \frac{N(\tau) -I}{\tau} \Pi_{\tau} u (s) \right\|_{L^{{q_0}'}(0,T; W^{1,{r_0}'})}
        + C\,\tau \left\| \frac{N(\tau) -I}{\tau} \Pi_{\tau} u (s) \right\|_{W^{1,{q_0}'}(0,T; L^{{r_0}'})} \\
    &\qquad= C \tau^{1/2} \left\| \mathcal{B}_1(u) \right\|_{L^{{q_0}'}(0,T; W^{1,{r_0}'})}
        + C \tau \left\| \mathcal{B}_1(u) \right\|_{W^{1,{q_0}'}(0,T; L^{{r_0}'})}
    \end{aligned}
    $$
for all admissible pair $(q,r)$.
On the other hand, for the second term of \eqref{eq-lem-K1}, we begin the proof by splitting
    $$
    \begin{aligned}
    \frac{N(\tau) - I}{\tau} \Pi_{\tau} u -i\lambda |u|^p u
    &= \Big( \frac{N(\tau) - I}{\tau} \Pi_{\tau} u -i\lambda |\Pi_{\tau}u|^p \Pi_{\tau}u \Big)
     + i\lambda\Big( |\Pi_{\tau}u|^p \Pi_{\tau}u - |u|^p u \Big) \\
    &= \mathcal{B}_2(u) + \mathcal{B}_3(u).
    \end{aligned}
    $$
Then, we can apply \eqref{eq-2-6} to obtain that
    $$
    \begin{aligned}
    &\left\| \int_{0}^{n\tau} S_{\tau}(n \tau -s) \frac{N(\tau) - I}{\tau} \Pi_{\tau} u(s)ds
        - i\lambda\int_0^{n\tau} S_{\tau}(n\tau -s) |u|^p u (s) ds \right\|_{\ell^q (0, T; L^r)}\\
    &\qquad\leq \left\| \int_{0}^{n\tau} S_{\tau}(n \tau -s) \mathcal{B}_2(u) ds \right\|_{\ell^q (0, T; L^r)}
        + \left\| \int_{0}^{n\tau} S_{\tau}(n \tau -s) \mathcal{B}_3(u) ds \right\|_{\ell^q (0, T; L^r)} \\
    &\qquad\leq C\big\| \mathcal{B}_2(u) \big\|_{L^{q_0'} (0, T; L^{r_0'})}
        + C\big\| \mathcal{B}_3(u) \big\|_{L^{q_0'} (0, T; L^{r_0'})}.
    \end{aligned}
    $$
for all admissible pair $(q,r)$.
Thus, it is enough to show that
    \begin{equation}\label{eq-lem-K111}
    \begin{aligned}
    &\tau^{1/2} \big\| \mathcal{B}_1(u) \big\|_{L^{{q_0}'}(0,T; W^{1,{r_0}'})}
    + \tau \big\| \mathcal{B}_1(u) \big\|_{W^{1,{q_0}'}(0,T; L^{{r_0}'})}
    + \big\| \mathcal{B}_2(u) \big\|_{L^{q_0'} (0, T; L^{r_0'})}
    + \big\| \mathcal{B}_3(u) \big\|_{L^{q_0'} (0, T; L^{r_0'})} \\
    &\qquad\qquad\leq C_{d,p}~ \tau^{1/2} \big( 1+ T^{\frac{1}{q_0'}} \big) \big( 1+ M_1^{2p+1} \big) .
    \end{aligned}
    \end{equation}

For the first term on the right hand side of \eqref{eq-lem-K111},
we apply Lemma \ref{lem-2-1}, H\"older's inequality \eqref{freq_Hol}, Lemma \ref{lem-2-6} and Sobolev's embedding  $H^1 (\mathbb{R}^d) \rightarrow L^{r_0} (\mathbb{R}^d)$ to obtain that
    $$
    \begin{aligned}
    \big\| \mathcal{B}_1(u) \big\|_{L^{{q_0}'}(0,T; W^{1,{r_0}'})}
    &= \left\| \frac{N(\tau) -I}{\tau} \Pi_{\tau} u (s) \right\|_{L^{q_0'}(0,T; W^{1,r_0'})} \\
    &\leq \Big\| |\Pi_{\tau} u|^{p+1} \Big\|_{L^{q_0'}(0,T; L^{r_0'})}
        + C_p \Big\| |\Pi_{\tau} u|^{p} |\nabla \Pi_{\tau} u| \Big\|_{L^{q_0'}(0,T; L^{r_0'})}
        \\
    &\leq C_p T^{\frac{1}{q_0'}-\frac{1}{q_0}} \big\| \Pi_{\tau}u \big\|_{L^{\infty} (0,T; L^{r_0})}^p
        \Big( \big\| \Pi_{\tau}u \big\|_{L^{q_0}(0,T; L^{r_0})} + \big\| \nabla \Pi_{\tau}u \big\|_{L^{q_0}(0,T; L^{r_0})} \Big) \\
    &\leq C_{d,p} T^{\frac{1}{q_0'}-\frac{1}{q_0}} \| u \|_{L^{\infty} (0,T; H^1)}^p \| u \|_{L^{q_0}(0,T; W^{1,r_0})},
    \end{aligned}
    $$
which is then bounded by $C_{d,p} T^{\frac{1}{q_0'}-\frac{1}{q_0}} M_1^{p+1}$, owing to Theorem \ref{wp_u_H12}.\\

Next, we estimate the second term on the right hand side of \eqref{eq-lem-K111}.
Similarly to \eqref{eq-2-13}, we deduce that
    \begin{equation}\label{eq-5-3}
    \begin{split}
    \big\| \mathcal{B}_1(u) \big\|_{W^{1,{q_0}'}(0,T; L^{{r_0}'})}
    &= \left\| \partial_t \bigg( \frac{N(\tau) -I}{\tau} \Pi_{\tau} u (s) \bigg) \right\|_{L^{q_0'}(0,T; L^{r_0'})}
    \\
    &\leq c_p \Big\| |\Pi_{\tau}u|^p \partial_t \Pi_{\tau} u (s) \Big\|_{L^{q_0'}(0,T; L^{r_0'})}
    \\
    &\leq  c_p \Big\| |\Pi_{\tau}u|^p  \Pi_{\tau} \Delta u  \Big\|_{L^{q_0'}(0,T; L^{r_0'})}
        +  c_p \Big\| |\Pi_{\tau}u|^p \Pi_{\tau} (|u|^p u)  \Big\|_{L^{q_0'}(0,T; L^{r_0'})},
    \end{split}
    \end{equation}
where the second inequality follows from the identities $\partial_t \Pi_{\tau} u = \Pi_{\tau} \partial_t u$ and $\partial_t u = i \Delta u + i\lambda |u|^{p} u$. By H\"older's inequality \eqref{freq_Hol} and Lemma \ref{lem-2-6}, the first term on the right hand side of \eqref{eq-5-3} is bounded by
    $$
    \begin{aligned}
    \Big\| |\Pi_{\tau}u|^p  \Pi_{\tau} \Delta u  \Big\|_{L^{q_0'}(0,T; L^{r_0'})}
    &\leq \big\| \Pi_{\tau}u \big\|_{L^{q_0}(0,T; L^{r_0})}^p \big\| \Pi_{\tau} \Delta u  \big\|_{L^{{q_0}}(0,T; L^{r_0})} \\
    &\leq C_{d,p} \tau^{-\frac{1}{2}} T^{\frac{1}{q_0'}-\frac{1}{q_0}} \big\| \Pi_{\tau}u \big\|_{L^{\infty}(0,T; L^{r_0})}^p \big\| \Pi_{\tau}\nabla u \big\|_{L^{{q_0}}(0,T; L^{r_0})},
    \end{aligned}
    $$
which is bounded by $C_{d,p} \tau^{-\frac{1}{2}} T^{\frac{1}{q_0'}-\frac{1}{q_0}} M_1^{p+1}$ owing to Theorem \ref{wp_u_H12}.
Also, for the second term of \eqref{eq-5-3}, we can apply Lemma \ref{LemSov_2p+1}, and the term is bounded by $C_{d,p}~ \tau^{-1/2} \big( 1 + T^{\frac{1}{q_0'}} \big) M_1^{2p+1}$.\\

For the third term of \eqref{eq-lem-K111}, we recall that $N(\tau) a = e^{i\tau \lambda |a|^{p}}a$ for $a \in \mathbb{C}$. Then using the mean value theorem, one has
    $$
    \big| \mathcal{B}_2(u) \big|
    = \bigg|\frac{N(\tau) - I}{\tau} \Pi_{\tau} u -i\lambda |\Pi_{\tau}u|^p \Pi_{\tau}u \bigg|
    \leq C \tau |\Pi_{\tau}u|^{2p+1}.
    $$
Thus, we have
    \begin{equation}
    \begin{split}
    \big\| \mathcal{B}_2(u) \big\|_{L^{q_0'} (0, T; L^{r_0'})}
    \leq C\tau \big\| |\Pi_{\tau} u|^{2p+1} \big\|_{L^{q_0'} (0, T; L^{r_0'})}
    \end{split}
    \end{equation}
Then, by Lemma \ref{LemSov_2p+1}, this term is bounded by $C_{d,p}~ \tau^{1/2} \big( 1 + T^{\frac{1}{q_0'}} \big) M_1^{2p+1}$.\\

For the last term of \eqref{eq-lem-K111}, we proceed as follows:
    $$
    \begin{aligned}
    \big\| \mathcal{B}_3(u) \big\|_{L^{q_0'} (0, T; L^{r_0'})}
    &=\Big\| |\Pi_{\tau}u|^p \Pi_{\tau}u - |u|^p u \Big\|_{L^{q_0'} (0, T; L^{r_0'})} \\
    &\leq C_{d,p} T^{\frac{1}{q_0'}-\frac{1}{q_0}} \big\| \Pi_{\tau}u - u \big\|_{L^{q_0} (0, T; L^{r_0})}
        \Big( \| \Pi_{\tau}u \|_{L^{\infty} (0, T; L^{r_0})}^p + \| u \|_{L^{\infty} (0, T; L^{r_0})}^p \Big)\\
    &\leq C_{d,p} \tau^{1/2} T^{\frac{1}{q_0'}-\frac{1}{q_0}} \| u \|_{L^{q_0} (0, T; W^{1,r_0})}  \| u \|_{L^{\infty} (0, T; H^1)}^p,
    \end{aligned}
    $$
where we have used Lemma \ref{lem-2-6} and Sobolev's embedding.
This is bounded by $C \tau^{1/2} T^{\frac{1}{q_0'}-\frac{1}{q_0}} M_1^{p+1}$ from Theorem \ref{wp_u_H12}.

Collecting the above estimates, we obtain \eqref{eq-lem-K111}. The proof is complete.
\end{proof}


\section{Global $H^1$ stability of $Z_{\tau}$ for $0<p<\frac{4}{d}$}\label{sec-4}

In this section, we firstly prove the local $H^1$ stability result on the scheme $Z_{\tau}$ in the space $\ell^{q}(n\tau \in I; W^{1,r}(\mathbb{R}^d))$ with initial data $\phi$ in $H^1 (\mathbb{R}^d)$ and $p \in (0, p_{d})$. After that, we will prove Theorem \ref{thm-9} for the case that $0<p<\frac{4}{d}$. The proof for this case will be derived by combining the global $L^2$ stability result (see Lemma \ref{prop-r-1}) and the local $H^1$ stability of $Z_{\tau}$ (see Proposition \ref{thm-3}).
\

By considering the discrete Strichartz estimate \eqref{eq-st-7}, we can take $\mathbf{C}=C(d,p) \in [1,\infty)$ such that
    \begin{equation}\label{def_C(d,p)}
    \mathbf{C}= \max\Big\{ \sup_{\tau\in (0,1)} \sup_{\phi \in H^1} \frac{\|S_{\tau} \phi \|_{\ell^{q_0}(\tau\mathbb{Z}, W^{1,r_0})} + \|S_{\tau} \phi \|_{\ell^{\infty}(\tau \mathbb{Z}, H^1)}}{ \|\phi\|_{{H}^1}}, ~1 \Big\}.
    \end{equation}
Then, we have the following proposition.

\begin{prop}[Local $H^1$ stability]\label{thm-3} Let $d \geq 1$ and $0<p< p_{d}$, and suppose that $\phi \in H^1 (\mathbb{R}^d)$. Then there exists a constant $\beta_{d,p}>0$ such that the $Z_{\tau}$ satisfies
    \begin{equation}\label{eq-1-1}
    \| Z_{\tau}(n \tau)\|_{\ell^q (0,T_0; W^{1,r} )}
    \leq 4 \mathbf{C} \|\phi\|_{H^1 (\mathbb{R}^d)} \quad\mbox{for all}\quad \tau \in (0,1),
    \end{equation}
where $(q,r) \in \{(q_0, r_0), (\infty,2)\}$  and $T_0>0$ is defined by
    \begin{equation}\label{eq-r-1}
    T_{0} = \beta_{d,p} \|\phi\|_{{H}^1}^{-\frac{2p(p+2)}{4-(d-2)p}}.
    \end{equation}
\end{prop}

\begin{proof}
To obtain the estimate \eqref{eq-1-1}, we consider the following set
    \begin{equation}\label{eq-r-31}
    \Lambda = \left\{ N \in \mathbb{N}\cup\{0\} ~:~
    \|Z_{\tau}(k\tau)\|_{\ell^{q_0} (0, N\tau ; W^{1,r_0})} + \| Z_{\tau} (k\tau)\|_{\ell^{\infty}(0, N\tau ; H^1)}
    \leq 4 \mathbf{C} \|\phi\|_{{H}^1}  \right\}.
    \end{equation}
If $\Lambda$ is an infinite set, then \eqref{eq-1-1} follows trivially.
Therefore, we suppose that $\Lambda$ is a finite set, and let $N_*$ be the largest element of $\Lambda$. It is then sufficient to find a lower bound on $N_*$, as the form of $N_* \geq T_0 /\tau$ for $T_0 >0$ defined in \eqref{eq-r-1} with a suitable choice of $\beta_{d,p}>0$.
\

First we verify that the set $\Lambda$ is non-empty.
Indeed, by the definitions of $Z_{\tau}$ and $\mathbf{C}$ given in \eqref{eq-r-30} and \eqref{def_C(d,p)} respectively, we find that
    $$
    \begin{aligned}
    \tau^{\frac{1}{q_0}} \| Z_{\tau}(0) \phi\|_{W^{1,r_0}} + \| Z_{\tau}(0) \phi\|_{H^{1}}
    &=\tau^{\frac{1}{q_0}} \| S_{\tau} (0) \phi \|_{W^{1,r_0}} +  \| S_{\tau} (0) \phi \|_{H^{1}} \\
    &\leq \| S_{\tau}(\tau \cdot) \phi \|_{\ell^{q_0}(\tau\mathbb{Z}; W^{1,r_0})} + \| S_{\tau}(\tau \cdot) \phi \|_{\ell^{\infty}(\tau\mathbb{Z}; H^{1})} \\
    &\leq \mathbf{C} \|\phi\|_{{H}^1},
    \end{aligned}
    $$
which means that $0 \in \Lambda$.
\

Let $(q,r)$ denote either $(q_0, r_0)$ or $(\infty,2)$. Then, using the Duhamel formula \eqref{eq-1-9} we have
    \begin{equation}\label{eq-3-2}
    \begin{split}
    &\left( \tau\sum_{n=0}^{N_*+1} \| Z_{\tau}(n \tau)\|_{W^{1,r}}^{q}\right)^{1/q} \\
    &\quad \leq \big\| S_{\tau}(n \tau) \phi \big\|_{\ell^{q} (0 \leq n \tau \leq (N_*+1)\tau; W^{1,r})}
        + \left\| \tau \sum_{k=0}^{n-1} S_{\tau} (n\tau -k\tau)
        \frac{N(\tau) - I}{\tau} Z_{\tau}(k \tau) \right\|_{\ell^{q} (\tau \leq n \tau \leq (N_* +1) \tau; W^{1,r} )} \\
    &\quad \leq \mathbf{C} \|\phi\|_{{H}^1}
        + \left\| \tau \sum_{k=0}^{n-1} S_{\tau} (n\tau -k\tau)
        \frac{N(\tau) - I}{\tau} Z_{\tau}(k \tau) \right\|_{\ell^{q} (\tau \leq n \tau \leq (N_* +1) \tau; W^{1,r} )},
    \end{split}
    \end{equation}
where \eqref{def_C(d,p)} is used for the second inequality.
We can bound the last term of \eqref{eq-3-2} by applying the Strichartz estimate \eqref{eq-st-8} as follows:
    \begin{equation}\label{def_C_2}
    \begin{aligned}
    &\left\| \tau \sum_{k=0}^{n-1} S_{\tau} (n\tau -k\tau) \frac{N(\tau) - I}{\tau} Z_{\tau}(k \tau) \right\|_{\ell^{q} (\tau \leq n \tau \leq (N_* +1) \tau; W^{1,r} )} \\
    &\qquad\leq C_{d,p} \left\| \frac{N(\tau) -I}{\tau} Z_{\tau}(n \tau) \right\|_{\ell^{q_0'} (0 \leq n \tau \leq N_* \tau; W^{1,r_0'})}.
    \end{aligned}
    \end{equation}
To estimate the right hand side of \eqref{def_C_2}, we apply Lemma \ref{lem-2-1} and H\"older's inequality \eqref{freq_Hol}. Then,
    $$
    \begin{aligned}
    &\left\| \frac{N(\tau) -I}{\tau} Z_{\tau}(n \tau) \right\|_{\ell^{q_0'} (0 \leq n \tau \leq N_* \tau; W^{1,r_0'})} \\
    &\quad \leq \left\| \frac{N(\tau) -I}{\tau} Z_{\tau}(n \tau) \right\|_{\ell^{q_0'} (0 \leq n \tau \leq N_* \tau; L^{r_0'})}
        + \left\| \nabla \left(\frac{N(\tau) -I}{\tau} Z_{\tau}(n \tau)\right) \right\|_{\ell^{q_0'} (0 \leq n \tau \leq N_* \tau; L^{r_0'})} \\
    &\quad \leq \big\| |Z_{\tau} (n\tau)|^{p+1} \big\|_{\ell^{q_0'} (0 \leq n \tau \leq N_* \tau; L^{r_0'})}
        + (p+1)\Big\| |Z_{\tau} (n\tau)|^p |\nabla Z_{\tau}(n\tau)| \Big\|_{\ell^{q_0'} (0 \leq n \tau \leq N_* \tau; L^{r_0'})} \\
    &\quad \leq (N_* \tau)^{\frac{1}{q_0'} - \frac{1}{q_0}} \| Z_{\tau} (n\tau) \|_{\ell^{\infty} (0 \leq n \tau \leq N_* \tau; L^{r_0})}^{p} \\
    &\qquad \times \Big( \| Z_{\tau}(n\tau) \|_{\ell^{q_0} (0 \leq n \tau \leq N_* \tau; L^{r_0})}
    + (p+1) \big\| \nabla Z_{\tau}(n\tau) \big\|_{\ell^{q_0} (0 \leq n \tau \leq N_* \tau; L^{r_0})} \Big).
    \end{aligned}
    $$
To proceed further, we utilize the Sobolev embedding $H^1 (\mathbb{R}^d) \rightarrow L^{r_0} (\mathbb{R}^d)$ and the fact that $N_* \in \Lambda$. Then we get
    \begin{equation}\label{eq-3-2-1}
    \begin{aligned}
    &\left\| \frac{N(\tau) -I}{\tau} Z_{\tau}(n \tau) \right\|_{\ell^{q_0'} (0 \leq n \tau \leq N_* \tau; W^{1,r_0'})} \\
    &\quad\leq C_{d,p} (N_* \tau)^{\frac{1}{q_0'} - \frac{1}{q_0}} \| Z_{\tau} (n\tau) \|_{\ell^{\infty} (0 \leq n \tau \leq N_* \tau; H^1)}^{p} \| Z_{\tau}(n\tau) \|_{\ell^{q_0} (0 \leq n \tau \leq N_* \tau; W^{1,r_0})} \\
    &\quad\leq C_{d,p} (N_*\tau)^{\frac{1}{q_0'}-\frac{1}{q_0}}  \Big( 4 \mathbf{C} \|\phi\|_{{H}^1} \Big)^{p+1} .
    \end{aligned}
    \end{equation}
Combining estimates \eqref{def_C_2} and \eqref{eq-3-2-1} in \eqref{eq-3-2}, we obtain
    $$
    \| Z_{\tau}(n \tau) \|_{\ell^{q} (0 \leq n \tau \leq (N_* +1) \tau ; W^{1,r} )}
    ~\leq~ \mathbf{C} \|\phi\|_{{H}^1} + C_{d,p} (N_*\tau)^{\frac{1}{q_0'}-\frac{1}{q_0}}  \Big( 4 \mathbf{C} \|\phi\|_{{H}^1} \Big)^{p+1}
    $$
for all $(q,r) \in \{(q_0, r_0), (\infty,2)\}$.
Consequently,
    \begin{equation}\label{eq-r-35}
    \begin{aligned}
    &    \| Z_{\tau}(n \tau)\|_{\ell^{\infty} (0\leq n\tau \leq (N_* +1) \tau; H^1)}
        + \| Z_{\tau}(n \tau) \|_{\ell^{q_0} (0 \leq n \tau \leq (N_* +1) \tau ; W^{1,r_0} )} \\
    &\qquad\leq 2 \mathbf{C} \|\phi\|_{{H}^1} + 2C_{d,p} (N_*\tau)^{\frac{1}{q_0'}-\frac{1}{q_0}}  \Big( 4\mathbf{C} \|\phi\|_{{H}^1} \Big)^{p+1}.
    \end{aligned}
    \end{equation}
This estimate yields that $N_*$ obeys the following estimate
    \begin{equation}\label{eq-r-32}
    2 C_{d,p} (N_*\tau)^{\frac{1}{q_0'} - \frac{1}{q_0}} \big( 4 \mathbf{C} \|\phi\|_{{H}^1} \big)^{p+1}
    \geq \mathbf{C} \|\phi\|_{{H}^1}.
    \end{equation}
Indeed, if \eqref{eq-r-32} does not hold, then it follows directly from \eqref{eq-r-35} that $N_* +1 \in \Lambda$, in view of definition \eqref{eq-r-31}. However, it is impossible by the maximality of $N_*$. Thus \eqref{eq-r-32} is true, and hence
\begin{equation*}
N_{*}\tau \geq \left( \frac{1}{2C_{d,p} 4^{p+1} \mathbf{C}^p \|\phi\|_{H_1}^p}\right)^{\frac{2(p+2)}{4-(d-2)p}},
\end{equation*}
where we  used the identity \eqref{eq-r-54}.
This shows that \eqref{eq-1-1} is true with the choice of $\beta_{d,p}= \big( 2 C_{d,p} 4^{p+1} \mathbf{C}^p  \big)^{-\frac{2(p+2)}{4-(d-2)p}}$ in \eqref{eq-r-1}.
The proof is finished.
\end{proof}

\begin{rem}\label{rem-r-4}
In the proof above, we note that the estimate \eqref{eq-1-1} holds for any admissible $(q,r)$, i.e.,
    $$
    \| Z_{\tau}(n \tau)\|_{\ell^q (0,T_0; W^{1,r} )}
    \leq C_q \|\phi\|_{H^1 (\mathbb{R}^d)}\quad\mbox{for all}\quad \tau \in (0,1).
    $$
Indeed, estimate \eqref{eq-3-2} holds for any admissible pair and its right hand side is bounded as in \eqref{def_C_2}.
\end{rem}

Before to show the global $H^1$ stability of $Z_{\tau}$ for the case of $0<p<\frac{4}{d}$,
we recall the $L^2$ stability result on $Z_{\tau}$ from \cite{I}.

\begin{lem}[Theorem 1.1 in \cite{I}, pages 3030--3032 in detail]\label{prop-r-1}
For $d\geq1$, $0 < p < \frac{4}{d}$ and $\phi \in L^2 (\mathbb{R}^d)$, there exist a constant $\widetilde{\beta}_{d, p} >0$ and a time $\widetilde{T}_0 = \widetilde{\beta}_{d,p} \|\phi\|_{L^2}^{-\frac{4p}{4-dp}}$ such that
    $$
    \|Z_{\tau}(n\tau)\|_{\ell^{q_0} (k\tau,k\tau + \widetilde{T}_0; L^{r_0})} \leq C_{d,p} \|\phi\|_{L^2}
    $$
holds  for all $k\in \mathbb{N}\cup \{0\}$.
\end{lem}

Now we are ready to prove the global $H^1$ stability of $Z_{\tau}$ for the mass-subcritical case.

\begin{proof}[Proof of Theorem \ref{thm-9}]
Consider $\phi \in H^1 (\mathbb{R}^d)$ and the solution $u$ to \eqref{eq-main} with initial data $\phi$. Let us set $T_1 >0$ by
    \begin{equation}\label{eq-r-52}
    T_1 := \gamma_{d,p} \min\big\{ T_0, \widetilde{T}_0 \big\}
    = \gamma_{d,p} \min\Big\{ \beta_{d,p} \|\phi\|_{H^1}^{-\frac{2p(p+2)}{4-(d-2)p}},~\widetilde{\beta}_{d, p} \|\phi\|_{L^2}^{-\frac{4p}{4-dp}}\Big\},
    \end{equation}
where a constant $\gamma_{d,p} \in(0,1)$ will be choose later.

In the case of $\tau \in [T_1 /2,1)$,
by H\"older's inequality in $t$ and \eqref{eq-2-21}, \eqref{eq-2-22} in Lemma \ref{lem-2-6}, we have
    \begin{equation*}
    \begin{aligned}
    \|Z_{\tau} (n \tau)\|_{\ell^{q} (n\tau \in [0,T]; W^{1,r})}
    &\leq \Big( \frac{2T}{\tau}\Big)^{\frac{1}{q}} \sup_{n\tau \in [0,T]} \|Z_{\tau}(n\tau)\|_{W^{1,r}} \\
    &\leq C_{d,q} T^{\frac{1}{q}} \tau^{-\frac{1}{q}} \tau^{-\frac{1}{2}} \tau^{\frac{d}{2} \left( \frac{1}{r} - \frac{1}{2}\right)} \sup_{n\tau \in [0,T]} \|Z_{\tau}(n\tau)\|_{L^2} \\
    &= C_{d,q} T^{\frac{1}{q}} \tau^{-\frac{1}{2}} \| \phi \|_{L^2}.
    \end{aligned}
    \end{equation*}
where we used also the definition of admissible pairs $\frac{2}{q} + \frac{d}{r}=\frac{d}{2}$ and the $L^2$ norm of $Z_{\tau}$ does not increase which obtained by Ignet \cite{I}, that is,
    \begin{equation*}
    \sup_{n\tau \in [0,\infty)} \|Z_{\tau}(n\tau)\|_{L^2}\leq \| \phi \|_{L^2} .
    \end{equation*}
Since $\tau \geq T_1$, we have
    \begin{equation*}
    \begin{aligned}
    \|Z_{\tau} (n \tau)\|_{\ell^{q} (n\tau \in [0,T]; W^{1,r})}
    \leq C_{d,q} T^{\frac{1}{q}} \tau^{-\frac{1}{2}} \Big( \frac{\tau}{T_1} \Big) \| \phi \|_{L^2} \\
    \leq C_{d,q} \tau^{1/2} T^{\frac{1}{q}} \big( \| \phi \|_{H^1} \big)^{C_{d,p}},
    \end{aligned}
    \end{equation*}
which proves the theorem for the case $\tau \in [T_1 /2,1)$.
\medskip

Now, we consider $\tau \in (0, T_1 /2)$ and we shoose a value $R \in (T_1 /2, T_1]$ such that $R/\tau \in \mathbb{N}$. We set $I_j = [jR, (j+1)R)$ for $j\in \mathbb{N} \cup \{0\}$.  For each $j \in \mathbb{N} \cup \{0\}$ we choose $m_j \in \mathbb{N}$ such that $m_j \tau = jR$, i.e., $m_j = j (R/\tau)$. Then
    \begin{equation}\label{eq-5-58}
    \| Z_{\tau}(n \tau) \|_{\ell^{q}(I_{j}; W^{1,r})} \leq \big\| Z_{\tau}(m_j \tau +n \tau) \big\|_{\ell^{q}(0,R; W^{1,r})}
    \end{equation}
for any admissible pair $(q,r)$. By applying the Strichartz estimates of Corollary \ref{thm-str2} to the Duhamel formula of $Z_{\tau} (m_j \tau +n\tau)$, we obtain
    \begin{equation}\label{eq-K1}
    \begin{split}
    &\| Z_{\tau}(m_j \tau +n \tau) \|_{\ell^{q}(0,R; W^{1,r})} \\
    &\quad\leq \big\| S_{\tau}(n \tau) Z_{\tau}(m_j \tau) \big\|_{\ell^{q} (0,R; W^{1,r})}
        + \left\| \tau \sum_{k=0}^{n-1} S_{\tau}(n \tau -k\tau) \frac{N (\tau) - I}{\tau} Z_{\tau}(m_j \tau + k \tau) \right\|_{\ell^{q} (0,R; W^{1,r})} \\
    &\quad\leq C_{d,q} \| Z_{\tau}(m_j \tau) \|_{{H}^1}
        + C_{d,q,p} \left\| \frac{N(\tau) -I}{\tau} Z_{\tau}(m_j \tau+ n \tau) \right\|_{\ell^{q_0'} (0,R; W^{1,r_0'})},
    \end{split}
    \end{equation}
where we can estimate the last term using Lemma \ref{lem-2-1} and H\"older's inequality as
    $$
    \begin{aligned}
    &\left\| \frac{N(\tau) -I}{\tau} Z_{\tau}(m_j \tau + n \tau) \right\|_{\ell^{q_0'} (0,R; W^{1,r_0'})} \\
    &\quad \leq \left\| \frac{N(\tau) -I}{\tau} Z_{\tau}(m_j \tau +n \tau) \right\|_{\ell^{q_0'} (0,R; L^{r_0'})}
        + \left\| \nabla \left(\frac{N(\tau) -I}{\tau} Z_{\tau}(m_j \tau +n \tau)\right) \right\|_{\ell^{q_0'} (0,R; L^{r_0'})} \\
    &\quad \leq L^{1-\frac{dp}{4}} \big\| Z_{\tau} (m_j \tau+n\tau) \big\|_{\ell^{q_0} (0,R; L^{r_0})}^{p} \\
    &\qquad\qquad \times \Big( \big\| Z_{\tau}(m_j \tau +n\tau) \big\|_{\ell^{q_0} (0,R; L^{r_0})} + (p+1) \big\| \nabla Z_{\tau}(m_j \tau +n\tau) \big\|_{\ell^{q_0} (0,R; L^{r_0})} \Big),
    \end{aligned}
    $$
where we used also the equalities
    \begin{equation*}
    \frac{1}{r_0'} = \frac{p+1}{r_0} \quad\mbox{and}\quad
    \frac{1}{q_0'} - \frac{(p+1)}{q_0} = 1-\frac{dp}{4}.
    \end{equation*}
By applying Lemma \ref{prop-r-1}, we estimate the right hand side as follows:
    $$
    \begin{aligned}
    &\left\| \frac{N(\tau) -I}{\tau} Z_{\tau}(m_j \tau + n \tau) \right\|_{\ell^{q_0'} (0,R; W^{1,r_0'})} \\
    &\quad \leq C_{d,p} \,L^{1-\frac{dp}{4}} \|\phi\|_{L^2}^{p}   \Big( \big\| Z_{\tau}(m_j \tau +n\tau) \big\|_{\ell^{q_0} (0,R; L^{r_0})} + \big\| \nabla Z_{\tau}(m_j \tau +n\tau) \big\|_{\ell^{q_0} (0,R; L^{r_0})} \Big)\\
    &\quad \leq C_{d,p} \Big( (2\gamma_{d,p} \widetilde{\beta}_{d, p})^{1-\frac{dp}{4}} \|\phi\|_{L^2}^{-p}\Big) \|\phi\|_{L^2}^{p}   \big\| Z_{\tau}(m_j \tau +n\tau) \big\|_{\ell^{q_0} (0,R; W^{1,r_0})}.
    \end{aligned}
    $$
Insert this estimate into \eqref{eq-K1}. Then, choosing $\gamma_{d,p}>0$ smaller in \eqref{eq-r-52} if necessary, we arrive at the following estimate:
    \begin{equation}\label{eq-c-20}
    \big\| Z_{\tau}(m_j \tau + n\tau) \big\|_{\ell^{q}(0,R; W^{1,r})}
    \leq C_{d,q} \big\| Z_{\tau}(n \tau) \big\|_{\ell^{\infty}(I_{j-1}; H^1)}
        + \frac{1}{2} \big\| Z_{\tau}(m_j \tau +n \tau) \big\|_{\ell^{q_0}(0,R; W^{1,r_0})}
    \end{equation}
    for any $j \in \mathbb{N}$.
This estimate with $(q,r) = (q_0, r_0)$ yields
    $$
    \big\| Z_{\tau}(m_j \tau + n\tau) \big\|_{\ell^{q_0}(0,R; W^{1,r_0})}
    \leq 2C_{d,p} \| Z_{\tau}(n \tau) \|_{L^{\infty}(I_{j-1}; H^1)}.
    $$
By inserting this back into \eqref{eq-c-20} and using \eqref{eq-5-58}, we obtain
    \begin{equation}\label{eq-c-21}
    \| Z_{\tau}(n \tau) \|_{\ell^{q}(I_{j}; W^{1,r})} \leq 2C_{d,p} \| Z_{\tau}(n \tau) \|_{\ell^{\infty}(I_{j-1}; H^1)},
    \end{equation}
for any admissible pair $(q,r)$ and $j \in \mathbb{N}$. By applying this with $(q,r)= (\infty,2)$, we finally deduce that for any $T>0$,
    $$
    \begin{aligned}
    \| Z_{\tau}(n \tau) \|_{\ell^{\infty}(0,T; H^{1})}
    & \leq \sum_{j=0}^{[2T/T_1]} \| Z_{\tau}(n \tau) \|_{\ell^{\infty}(I_j; H^{1})}
    \\
    &\leq \Biggl(\sum_{j=0}^{[2T/T_1]} (2C_{d,p})^j\Biggr) \| Z_{\tau}(n \tau) \|_{\ell^{\infty}(I_0; H^1)}
    \\
    &\leq \exp\Big( \frac{2T}{T_1} \log(2C_{d,p}) \Big) \|\phi \|_{H^1},
    \end{aligned}
    $$
where we have used Proposition \ref{thm-3} for the final inequality. Combining this with \eqref{eq-c-21} we complete the proof of Theorem \ref{thm-9}.
Indeed, we get an upper bound
    $$
    \| Z_{\tau}(n \tau) \|_{\ell^{q}(0,T; W^{1,r})}
    \leq \exp\Big( C_{d,p}T  \max\big\{ \|\phi\|_{H^1}^{\frac{2p(p+2)}{4-(d-2)p}}, ~\|\phi\|_{L^2}^{\frac{4p}{4-dp}}\big\} \Big).
    $$
The proof is finished.
\end{proof}


\section{More lemmas for the case of $p \geq1$.}\label{sec-10'}

In Section \ref{sec-10'}-\ref{sec-10}, we precisely write $Z_{\tau}^{\phi}$ and $u^{\phi}$ to denote the flow $Z_{\tau}$ and the solution $u$ corresponding to the initial data $\phi$. This notation will help to make clear our argument.
Firstly, we introduce well-known well-posedness theory on \eqref{eq-main} for $p \geq 1$ as follows.

\begin{theoremalpha}\label{wp_u_H12'}
[Theorem 5.3.1 in \cite{Ca}.] Let $d \geq 1$, $1 \leq p < p_{d}$, and $(q,r)$ be any admissible pair.
For any quantity $M \geq 1$, there is a time $T_0 = c_{d,p} M^{-\frac{2p(p+2)}{4-(d-2)p}}$ with some absolute constant $c_{d,p}>0$ such that the following statements hold:
    \begin{itemize}
    \item If $\phi_1, \phi_2 \in H^1 (\mathbb{R}^d)$ with $\| \phi_1 \|_{H^1}, \| \phi_1 \|_{H^1} \leq M$, there is a constant $C_{d,p}>0$ such that
    \begin{equation}\label{eq-c-1}
    \|u^{\phi_1}-u^{\phi_2}\|_{L^{q}(0,T_0; W^{1,r})}
    \leq C_{d,p} \|\phi_1 - \phi_2\|_{H^1}.
    \end{equation}
    \item If $\psi \in H^2 (\mathbb{R}^d)$ with $\| \psi \|_{H^1} \leq M$, there is a constant $M_3 =M_3(d,p,M,\psi)> 0$ such that
    \begin{equation}\label{def_M_2}
    \|u^{\psi}\|_{L^\infty (0,T_0; H^2)} + \|u^{\psi}\|_{L^q (0,T_0; W^{2,r})}
    \leq M_3.
    \end{equation}
    \end{itemize}
\end{theoremalpha}

In the following result, we obtain stability of $Z_{\tau}^{\phi}$ similarly to that of $u^{\phi}$ given in Theorem \ref{wp_u_H12'}. It will be essential for the global $H^1$ stability of $Z_\tau$ for the energy-subcritical case.

\begin{prop}\label{prop-3-3}
Let $d \geq 1$, $1 \leq p < p_{d}$, and $(q,r)$ be any admissible pair.
For any $M \geq 1$, there is a constant $\beta_{d,p}>0$
    \begin{equation}\label{def_T_*}
    T_2 := \beta_{d,p} M^{-\frac{4p(p+2)}{4-(d-2)p}} (< T_0)
    \end{equation}
such that the following statements hold:
    \begin{itemize}
    \item If $\phi_1, \phi_2 \in H^1 (\mathbb{R}^d)$ with $\| \phi_1 \|_{H^1}, \| \phi_1 \|_{H^1} \leq M$, there is a constant $C_{d,p} \geq 1$ such that
    \begin{equation}\label{eq-K-10}
    \big\| Z_{\tau}^{\phi_1}(n \tau) - Z_{\tau}^{\phi_2}(n \tau) \big\|_{\ell^q (0, T_2; W^{1,r} )}
    \leq C_{d,p} \|\phi_1 - \phi_2\|_{H^1 (\mathbb{R}^d)}
    \end{equation}
    for any $\tau \in (0,1)$.
    \item If $\psi \in H^2 (\mathbb{R}^d)$ with $\| \psi \|_{H^1} \leq M$, there is a constant $C(d,p,M,\psi)> 0$ such that
    \begin{equation}\label{thm-7}
    \big\|Z_{\tau}^{\psi}(n \tau)- u^{\psi} (n\tau) \big\|_{\ell^q (0, T_2; W^{1,r} )} \leq \tau^{1/2} C(d,p,M,\psi)
    \end{equation}
    for any $\tau \in (0,1)$.
    \end{itemize}
\end{prop}
The proof of \eqref{eq-K-10} is motivated by  the proof of \eqref{eq-c-1}, where we  apply the Duhamel-type formular \eqref{eq-1-9} for $Z_\tau^{\phi}$ instead of the Duhamel-type formular \eqref{eq-1-8} for $u^{\phi}$.
The idea of the proof of \eqref{thm-7} is  similar to that of the proof of \eqref{rel-goal} in Theorem \ref{thm-rel}.

To establish Proposition \ref{prop-3-3}, we introduce a technical lemma.
\begin{lem}\label{lem-a-1}
Suppose that $d \geq 1$ and $1 \leq p < p_{d}$, and let $(q_0, r_0)$ denote the admissible pair $(\frac{4(p+2)}{dp}, p+2)$. Then, for any time interval $I$ and functions $v,w: I \times \mathbb{R}^{d}\rightarrow\mathbb{C}$, we have
    $$
    \begin{aligned}
    &\left\| \frac{N(\tau)-I}{\tau} v - \frac{N(\tau)-I}{\tau} w \right\|_{\ell^{q_0'} (I; W^{1,r_0'})} \\
    &\quad\leq C |I|^{\frac{1}{q_0'} - \frac{1}{q_0}} \| v-w\|_{\ell^{\infty} (I; H^1)} \biggl[\Big( \| v \|_{\ell^{q_0} (I; W^{1,r_0})} + \| w \|_{\ell^{q_0} (I; W^{1,r_0})} \Big) \Big( \|v\|_{\ell^{\infty} (I; H^1)}^{p-1} + \|w\|_{\ell^{\infty} (I; H^1)}^{p-1} \Big) \Bigr.
    \\
    &\qquad\qquad\qquad\qquad\qquad\qquad\qquad\qquad\qquad +\Bigl. \tau \|\nabla w\|_{\ell^{q_0}(I; L_x^{r_0})} \big\| |w|^{2p-1} \big\|_{\ell^{\infty}(I; L^{\frac{r_0}{p-1}})}\biggr].
    \end{aligned}
    $$
\end{lem}

\begin{proof}
The proof basically follows from Lemma \ref{lem-2-1} and estimates in \cite[Section 4.4]{Ca}. Firstly, we apply \eqref{eq-2-7} and H\"older's inequality, to deduce that
    $$
    \begin{aligned}
    &\left\| \frac{N(\tau)-I}{\tau} v - \frac{N(\tau)-I}{\tau} w \right\|_{\ell^{q_0'} (I; L^{r_0'})} \\
    &\quad\leq C |I|^{\frac{1}{q_0'} - \frac{1}{q_0}} \| v-w\|_{\ell^{\infty} (I; L^{r_0})} \Big( \|v\|_{L^{q_0} (I; L^{r_0})} \|v\|_{\ell^{\infty} (I; L^{r_0})}^{p-1} + \|w\|_{L^{q_0} (I; L^{r_0})} \|w\|_{\ell^{\infty} (I; L^{r_0})}^{p-1} \Big) \\
    &\quad\leq C |I|^{\frac{1}{q_0'} - \frac{1}{q_0}} \| v-w\|_{\ell^{\infty} (I; H^1)} \Big( \|v\|_{\ell^{q_0} (I; L^{r_0})} \|v\|_{\ell^{\infty} (I; H^1)}^{p-1} + \|w\|_{\ell^{q_0} (I; L^{r_0})} \|w\|_{\ell^{\infty} (I; H^1)}^{p-1} \Big),
    \end{aligned}
    $$
where the Sobolev embedding is used for the last inequality, with the fact that $r_0 = p+2 < p_d +2$.

Next, by differentiating and rearranging, we have that
    $$
    \begin{aligned}
    &\nabla \left( \frac{N(\tau)-I}{\tau} v\right) -\nabla \left( \frac{N(\tau)-I}{\tau} w\right) \\
    &\quad= \nabla \left( \frac{e^{i\tau \lambda |v|^p}-1}{\tau} v\right) -\nabla \left( \frac{e^{i\tau \lambda |w|^p}-1}{\tau} w\right) \\
    &\quad= i \lambda p\Bigl(e^{i\tau \lambda |v|^{p}}|v|^{p} \nabla v -e^{i\tau \lambda |w|^p}  |w|^p \nabla w  \Bigr) \\
    &\qquad\qquad+ \left[\left( \frac{e^{i\tau \lambda |v|^p}-1}{\tau} \right) -\left( \frac{e^{i\tau \lambda |w|^p}-1}{\tau} \right) \right]\nabla v + \left( \frac{e^{i\tau \lambda |w|^p}-1}{\tau} \right)( \nabla v - \nabla w),
    \end{aligned}
    $$
which together with Lemma \ref{lem-2-1} yields that
    \begin{equation}\label{eq-a-11}
    \begin{split}
    &\biggl\| \nabla \left( \frac{N(\tau)-I}{\tau} v\right) -\nabla \left( \frac{N(\tau)-I}{\tau} w\right) \biggr\|_{\ell^{q_0'}(I; L^{r_0'})} \\
    &\qquad \leq p\Big\| e^{i\tau \lambda |v|^{p}}|v|^{p} \nabla v - e^{i\tau \lambda |w|^{p}} |w|^p \nabla w \Big\|_{\ell^{q_0'}(I; L^{r_0'})} \\
    &\qquad\qquad + \Big\|\big( |v|^{p}- |w|^{p} \big) |\nabla v|\Big\|_{\ell^{q_0'}(I; L^{r_0'})}
    + \Big\| |w|^{p}| \nabla w - \nabla v| \Big\|_{\ell^{q_0'}(I; L^{r_0'})} .
    \end{split}
    \end{equation}
Using the following identity
    \begin{equation}\label{eq-r-12}
    \begin{aligned}
    &e^{i\tau \lambda |v|^{p}}|v|^{p} \nabla v - e^{i\tau \lambda |w|^{p}} |w|^p \nabla w \\
    &\quad= \big( |v|^{p} \nabla v - |w|^p \nabla w \big) e^{i\tau \lambda |v|^p} +  |w|^p \nabla w \left( e^{i\tau \lambda |v|^p} - e^{i\tau \lambda|w|^p}\right) ,
    \end{aligned}
    \end{equation}
we can decompose the first term on the right hand side of \eqref{eq-a-11} into
    \begin{equation}\label{eq-r-90}
    \begin{aligned}
    &\big\| e^{i\tau \lambda |v|^{p}}|v|^{p} \nabla v - e^{i\tau \lambda |w|^{p}} |w|^p \nabla w \big\|_{\ell^{q_0'} (I; L^{r_0'})} \\
    &\qquad \leq \big\| |v|^p (\nabla v - \nabla w) \big\|_{\ell^{q_0'} (I; L^{r_0'})}
    + \big\| (|v|^p - |w|^p) \nabla w \big\|_{\ell^{q_0'} (I; L^{r_0'})} \\
    &\qquad\qquad + \Big\| |w|^p \nabla w \big( e^{i\tau \lambda |v|^p} - e^{i\tau \lambda|w|^p} \big) \Big\|_{\ell^{q_0'} (I; L^{r_0'})}.
    \end{aligned}
    \end{equation}
From H\"older's inequality and the Sobolev embedding, we estimate the first right term of \eqref{eq-r-90} as
    $$
    \begin{aligned}
    \big\| |v|^p (\nabla v - \nabla w) \big\|_{\ell^{q_0'}(I; L^{r_0'})}
    &\leq |I|^{\frac{1}{q_0'}} \big\| |v|^p \big\|_{\ell^{\infty} (I; L^{r_0/p})} \big\| \nabla v- \nabla w\|_{\ell^{\infty} (I; L^2)} \\
    &\leq C |I|^{\frac{1}{q_0'}} \|v\|_{\ell^{\infty} (I; H^1)}^p \|\nabla v- \nabla w\|_{\ell^{\infty} (I; L^2)} .
    \end{aligned}
    $$
Also, since $p \geq 1$, we have an inequality
    $$
    \big| |v|^p - |w|^p \big|
    \leq C|v-w| \big(|v|^{p-1} + |w|^{p-1} \big)\quad \forall~v,w \in \mathbb{C}.
    $$
Using this we estimate the second right term of \eqref{eq-r-90} as
    \begin{equation}\label{eq-r-91}
    \begin{aligned}
    & \big\| (|v|^p - |w|^p) \nabla w \big\|_{\ell^{q_0'} (I; L^{r_0'})}\\
    &\quad\leq |I|^{\frac{1}{q_0'} - \frac{1}{q_0}} \| v-w\|_{\ell^{\infty} (I; L^{r_0})} \Big\| |v|^{p-1} + |w|^{p-1} \Big\|_{\ell^{\infty} (I; L^{\frac{r_0}{p-1}})} \| \nabla w \|_{\ell^{q_0} (I; L^{r_0})} \\
    &\quad\leq C |I|^{\frac{1}{q_0'} - \frac{1}{q_0}} \| v-w\|_{\ell^{\infty} (I; H^1)} \| \nabla w \|_{\ell^{q_0} (I; L^{r_0})} \Big( \|v\|_{\ell^{\infty} (I; L^{r_0})}^{p-1} + \|w\|_{\ell^{\infty} (I; L^{r_0})}^{p-1} \Big) .
    \end{aligned}
    \end{equation}
For the third term, we notice that
    \begin{equation*}
    \Big| |w|^p \nabla w \left( e^{i\tau \lambda |v|^p} - e^{i\tau \lambda|w|^p}\right)\Big|
    \leq C \tau|w|^{p} |\nabla w|\left(|v|^{p-1} + |w|^{p-1} \right) |v-w|.
    \end{equation*}
Then, similarly to \eqref{eq-r-91} we obtain
    \begin{equation*}
    \begin{aligned}
    &\Big\| |w|^p \nabla w \big( e^{i\tau \lambda |v|^p} - e^{i\tau \lambda|w|^p} \big) \Big\|_{\ell^{q_0'} (I; L^{r_0'})} \\
    &\qquad\leq  C \tau |I|^{\frac{1}{q_0'} - \frac{1}{q_0}} \|v-w\|_{\ell^{\infty} (I; H^1)} \|\nabla w \|_{\ell^{q_0}(I; L_x^{r_0})} \left\| (|v|+|w|)^{2p-1}\right\|_{\ell^{\infty}(I; L^{\frac{r_0}{p-1}})}.
    \end{aligned}
    \end{equation*}
Combination of the estimates above gives the desired bound for the first term of \eqref{eq-a-11}, with help of the Sobolev embedding $H^1 (\mathbb{R}^d) \hookrightarrow L^{r_0} (\mathbb{R}^d)$.
The second and third terms of the right hand side of \eqref{eq-a-11} can be estimated as we did for the right hand sides of \eqref{eq-r-90}. The proof is done.
\end{proof}

Now, we ready to show Proposition \ref{prop-3-3}.

\begin{proof}[Proof of Proposition \ref{prop-3-3}]
In the proof, we utilize the admissible pair $(q_0, r_0)=(\frac{4(p+2)}{dp}, p+2)$. We divide the proof into two parts corresponding to \eqref{eq-K-10} and \eqref{thm-7}. \\

\textbf{Proof of \eqref{eq-K-10}.}
Let $\phi_1$ and $\phi_2 \in H^1 (\mathbb{R}^d)$ such that $\|\phi_1\|_{H^1} \leq M$ and $\|\phi_2\|_{H^1} \leq M$. We consider the difference between the Duhamel formulas of $Z_{\tau}^{\phi_1}$ and $Z_{\tau}^{\phi_2}$ provided by \eqref{eq-1-9}.
Then by applying the Strichartz estimate \eqref{eq-st-6}, we have
    \begin{equation}\label{eq-5-20}
    \begin{split}
    &\big\| Z_{\tau}^{\phi_1} (n \tau)-Z_{\tau}^{\phi_2} (n \tau) \big\|_{\ell^q (0, T_2; W^{1,r} )} \\
    &\quad= \left\| S_{\tau}(n \tau) (\phi_1 - \phi_2)
        ~+~ \tau \sum_{k=0}^{n-1} S_{\tau} \big((n-k) \tau \big)
       \biggl( \frac{N(\tau) - I}{\tau} Z_{\tau}^{\phi_1}- \frac{N(\tau) - I}{\tau}Z_{\tau}^{\phi_2}\biggr)(k \tau) \right\|_{\ell^q (0, T_2; W^{1,r} )} \\
    &\quad\leq C\| \phi_1 - \phi_2 \|_{H^1 (\mathbb{R}^d)}
        + C \left\| \frac{N(\tau) - I}{\tau} Z_{\tau}^{\phi_1}- \frac{N(\tau) - I}{\tau}Z_{\tau}^{\phi_2} \right\|_{\ell^{q_0'} (0, T_2; W^{1,r_0'} )}
    \end{split}
    \end{equation}
     for any admissible pair $(q,r)$.
By the local $H^1$ stability of Proposition \ref{thm-3} with choosing $\beta_{d,p} >0$ in \eqref{def_T_*} small enough, for $(q,r) \in \{(q_0, r_0), (\infty,2)\}$, we have
    \begin{equation}\label{eq-r-10}
    \| Z_{\tau}^{\phi_j}(n \tau)\|_{\ell^q (0,T_2; W^{1,r} )}
    \leq \| Z_{\tau}^{\phi_j}(n \tau)\|_{\ell^q (0,T_0; W^{1,r} )}
    \leq 4 \mathbf{C}M, \quad j=1,2,
    \end{equation}
since $T_2 <T_0$.
We proceed to use bound \eqref{eq-r-10} to estimate the last term of \eqref{eq-5-20}. By applying Lemma \ref{lem-a-1} and the Sobolev embedding in \eqref{eq-5-20} along with \eqref{eq-r-10}, we obtain that
    \begin{equation}\label{eq-3-10}
    \begin{split}
    &\big\| Z_{\tau}^{\phi_1} (n \tau) -Z_{\tau}^{\phi_2} (n \tau) \big\|_{\ell^q (0, T_2; W^{1,r} )} \\
    &\leq  C \| \phi_1 - \phi_2 \|_{H^1 (\mathbb{R}^d)} \\
    &\quad +  C T_2^{\frac{1}{q_0'} - \frac{1}{q_0}}
        \big\| Z_{\tau}^{\phi_1} (n \tau) -Z_{\tau}^{\phi_2} (n \tau) \big\|_{\ell^{\infty} (0, T_2; H^1 )}
        \bigg( (4\mathbf{C}M)^p + \tau (4\mathbf{C}M) \big\| |Z_{\tau}^{\phi_2}|^{2p-1} \big\|_{\ell^{\infty} (0,T_2; L^{\frac{r_0}{p-1}})} \bigg).
    \end{split}
    \end{equation}
To estimate the \eqref{eq-3-10}, we apply \eqref{eq-2-22} in Lemma \ref{lem-2-6} to get
    \begin{equation*}
    \|Z_{\tau}^{\phi_2}(n\tau)\|_{L^{\frac{(2p-1)(p+2)}{p-1}}}^{2p-1}   \leq C \tau^{-\frac{d}{2}\left(\frac{1}{p+2} - \frac{1}{p+2} \frac{(p-1)}{(2p-1)}\right) (2p-1)} \|Z_{\tau}^{\phi_2}(n\tau)\|_{L^{p+2}}^{2p-1}.
    \end{equation*}
Here, $\frac{d}{2}\big(\frac{1}{p+2} - \frac{1}{p+2} \frac{(p-1)}{(2p-1)}\big) (2p-1) = \frac{dp}{2(p+2)} < 1$ since $p < p_d$. Therefore, for $\tau \in (0,1)$, the estimate above together with the Sobolev embedding and \eqref{eq-r-10} yields
    \begin{equation}\label{eq-r-36}
    \begin{aligned}
    \tau \|Z_{\tau}^{\phi_2}\|_{\ell^{\infty} \big( 0,T_2;L^{\frac{(2p-1)(p+2)}{p-1}} \big)}^{2p-1}
    &\leq C \tau^{1- \frac{dp}{2(p+2)}} \|Z_{\tau}^{\phi_2}\|_{\ell^{\infty}(0,T_2;L^{p+2})}^{2p-1} \\
    &\leq C \tau^{1- \frac{dp}{2(p+2)}} \|Z_{\tau}^{\phi_2}\|_{\ell^{\infty}(0,T_2;H^1)}^{2p-1}
    \leq C_{d,p} (4\mathbf{C} M)^{2p-1}.
    \end{aligned}
    \end{equation}
Inserting this inequality \eqref{eq-r-36} into \eqref{eq-3-10}, we have
    \begin{equation*}
    \begin{split}
    &\big\| Z_{\tau}^{\phi_1} (n \tau) -Z_{\tau}^{\phi_2} (n \tau) \big\|_{\ell^q (0, T_2; W^{1,r} )} \\
    &\quad\leq  C \| \phi_1 - \phi_2 \|_{H^1 (\mathbb{R}^d)}
     +  C_{d,p} T_2^{\frac{1}{q_0'} - \frac{1}{q_0}} (4\mathbf{C} M)^{2p}
        \big\| Z_{\tau}^{\phi_1} (n \tau) -Z_{\tau}^{\phi_2} (n \tau) \big\|_{\ell^{\infty} (0, T_2; H^1 )}.
    \end{split}
    \end{equation*}
Now, we choose $T_2 = \big[ (2 C_{d,p}) (4\mathbf{C} M)^{2p} \big]^{\frac{q_0 q_0'}{q_0' -q_0}}$. Then, the estimate above with $(q,r)= (\infty, 2)$ yields that
    $$
    \big\| Z_{\tau}^{\phi_1} (n \tau) -Z_{\tau}^{\phi_2} (n \tau) \big\|_{\ell^{\infty} (0, T_2; H^{1} )}
    \leq  2C \| \phi_1 - \phi_2 \|_{H^1 (\mathbb{R}^d)}.
    $$
By inserting this into \eqref{eq-3-10} for general pairs $(q,r)$, we obtain \eqref{eq-K-10}.

\
\

\textbf{Proof of \eqref{thm-7}.}
Assume that $\psi \in H^2 (\mathbb{R}^d)$ with $\|\psi\|_{H^1}\leq M$.
For our aim, it is sufficient to estimate $Z_{\tau}^{\psi}(n \tau) - \Pi_{\tau} u^{\psi}(n \tau)$ instead of $Z_{\tau}^{\psi}(n \tau) - u^{\psi}(n \tau)$, because we have
    $$
    \begin{aligned}
    \big\| u^{\psi}(n \tau) - \Pi_{\tau} u^{\psi}(n \tau) \big\|_{\ell^{q} (0,T_2; W^{1,r})}
    &\leq C\tau^{1/2} \| u^{\psi} (n \tau)\|_{\ell^{q} (0,T_2; W^{2,r})} \\
    &\leq C\tau^{1/2} M_3 = \tau^{1/2} C(d,p,M,\psi)
    \end{aligned}
    $$
thanks to \eqref{eq-2-20} and \eqref{def_M_2}.
By utilizing the Duhamel formulas \eqref{eq-1-8} and \eqref{eq-1-9}, we decompose $Z_{\tau}^{\psi}(n\tau) - \Pi_{\tau}u^{\psi} (n\tau)$ into
    \begin{equation}\label{eq-6-1}
    \begin{split}
    &Z_{\tau}^{\psi}(n \tau) - \Pi_{\tau}u^{\psi} (n \tau) \\
    &\quad=\tau \sum_{k=0}^{n-1} S_{\tau}(n \tau -k \tau) \left( \frac{N(\tau) - I}{\tau} Z_{\tau}^{\psi}(k\tau)
        - \frac{N(\tau) - I}{\tau} \Pi_{\tau} u^{\psi}(k\tau) \right) \\
    &\qquad +\tau \sum_{k=0}^{n-1} S_{\tau}(n \tau -k \tau) \frac{N(\tau) - I}{\tau} \Pi_{\tau} u^{\psi}(k\tau)
        - i\lambda\int_0^{n\tau} S_{\tau}(n\tau -s) |u^{\psi}|^p u^{\psi} (s) ds.
    \end{split}
    \end{equation}
Let $(q,r)$ be an admissible pair. We proceed to find an estimate on the $\ell^{q} (0,T_2 ;W^{1,r})$ norm of $Z_{\tau}^{\psi} - \Pi_{\tau} u^{\psi}$ using decomposition \eqref{eq-6-1}. Firstly, we estimate the first term on the right hand side of \eqref{eq-6-1}.
For this, by applying the Strichartz estimate \eqref{eq-st-6} and Lemma \ref{lem-a-1} in order, we arrive at the following estimate
    $$
    \begin{aligned}
    &\left\| \tau \sum_{k=0}^{n-1} S_{\tau}(n \tau -k \tau) \left( \frac{N(\tau) - I}{\tau} Z_{\tau}^{\psi}(k\tau)
        - \frac{N(\tau) - I}{\tau} \Pi_{\tau} u^{\psi} (k \tau) \right) \right\|_{\ell^{q} (0,T_2; W^{1,r})} \\
    &\quad\leq C \left\| \frac{N(\tau) - I}{\tau} Z_{\tau}^{\psi}(k\tau)
        - \frac{N(\tau) - I}{\tau} \Pi_{\tau} u^{\psi} (k \tau) \right\|_{\ell^{q_0 '} (0,T_2; W^{1,r_0'} )} \\
    &\quad\leq C T_2^{\frac{1}{q_0'} - \frac{1}{q_0}} \| Z_{\tau}^{\psi} - \Pi_{\tau} u^{\psi} \|_{\ell^{\infty} (0,T_2; H^1)} \biggl[\Big( \| Z_{\tau}^{\psi} \|_{\ell^{q_0} (0,T_2; W^{1,r_0})} + \| \Pi_{\tau} u^{\psi} \|_{\ell^{q_0} (0,T_2; W^{1,r_0})} \Big)\biggr. \\
    &\biggl.\qquad \times\Big( \|Z_{\tau}^{\psi}\|_{\ell^{\infty} (0,T_2; H^1)}^{p-1} + \| \Pi_{\tau} u^{\psi} \|_{\ell^{\infty} (0,T_2; H^1)}^{p-1} \Big) + \tau \|\nabla Z_{\tau}^{\psi}\|_{\ell^{q_0}(0,T_2 ;L_x^{r_0})} \Bigl\| ||Z_{\tau}^{\psi}|^{2p-1}\Bigr\|_{\ell^{\infty}(0,T_2; L^{\frac{r_0}{p-1}})}\biggr].
    \end{aligned}
    $$
On the other hand, we recall from Proposition \ref{thm-3} with the $T_2 \leq T_0$ given in \eqref{def_T_*} that $Z_{\tau}^{\psi}$ enjoys the following local stability
    \begin{equation}\label{add_est_1}
    \| Z_{\tau}^{\psi} \|_{\ell^{\infty} (0,T_2; H^1)}  \leq 4\mathbf{C}M \quad\mbox{and}\quad
    \| Z_{\tau}^{\psi} \|_{\ell^{q_0} (0,T_2; W^{1,r_0})} \leq 4\mathbf{C}M.
    \end{equation}
Also, from \eqref{eq-r-36}, we see that
    \begin{equation}\label{add_est_2}
    \tau  \|\nabla Z_{\tau}^{\psi}\|_{\ell^{q_0}(0,T_2;L^{r_0})} \bigl\| ||Z_{\tau}^{\psi}|^{2p-1}\bigr\|_{\ell^{\infty}(0,T_2; L^{\frac{r_0}{p-1}})}
    \leq C_{d,p} (4\mathbf{C}M)^{2p}.
    \end{equation}
Furthermore, by \eqref{eq-2-20'} in Lemma \ref{lem-2-6} and \eqref{eq-c-1} in Theorem \ref{wp_u_H12'} with $\psi_1 = \psi$ and $\psi_2 =0$, we have
    \begin{equation}\label{add_est_3}
    \| \Pi_{\tau} u^{\psi} \|_{\ell^{\infty} (0,T_2; H^1)} + \| \Pi_{\tau} u^{\psi} \|_{\ell^{q_0} (0,T_2; W^{1,r_0})} \leq C_{d,p}M.
    \end{equation}
By applying estimates \eqref{add_est_1}, \eqref{add_est_2} and \eqref{add_est_3}, we get
    $$
    \begin{aligned}
    &\left\| \tau \sum_{k=0}^{n-1} S_{\tau}(n \tau -k \tau) \left( \frac{N(\tau) - I}{\tau} Z_{\tau}^{\psi}(k\tau)
        - \frac{N(\tau) - I}{\tau} \Pi_{\tau} u^{\psi} (k \tau) \right) \right\|_{\ell^{q} (0,T_2; W^{1,r})} \\
    &\qquad\leq C_{d,p} T_2^{\frac{1}{q_0'} - \frac{1}{q_0}} (4\mathbf{C}M)^{2p} \| Z_{\tau}^{\psi} - \Pi_{\tau} u^{\psi} \|_{\ell^{\infty} (0,T_2; H^1)} \\
    &\qquad\leq \frac{1}{2} \| Z_{\tau}^{\psi} - \Pi_{\tau} u^{\psi} \|_{\ell^{\infty} (0,T_2; H^1)}
    \end{aligned}
    $$
provided $\beta_{d,p}>0$ in $T_2$ is small enough in \eqref{def_T_*}. By inserting this estimate into \eqref{eq-6-1}, we obtain
    $$
    \begin{aligned}
    &\big\| Z_{\tau}^{\psi} - \Pi_{\tau} u^{\psi} \big\|_{\ell^{q} (0,T_2; W^{1,r})} \\
    &\quad\leq 2\left\| \tau \sum_{k=0}^{n-1} S_{\tau}(n \tau -k \tau) \frac{N(\tau)-I}{\tau} \Pi_{\tau} u^{\psi} (k \tau)
    - i\lambda\int_0^{n\tau} S_{\tau}(n\tau -s) |u^{\psi}|^p u^{\psi} (s) ds \right\|_{\ell^q (0,T_2; W^{1,r})} .
    \end{aligned}
    $$
On the other hand, one has
    \begin{equation}\label{eq-r-53}
    \begin{split}
    \left\| \tau \sum_{k=0}^{n-1} S_{\tau}(n \tau -k \tau) \frac{N(\tau)-I}{\tau} \Pi_{\tau} u^{\psi} (k \tau)
    - i\lambda\int_0^{n\tau} S_{\tau}(n\tau -s) |u^{\psi}|^p u^{\psi} (s) ds \right\|_{\ell^q (0,T_2; W^{1,r})} \\
    \leq C_{d,p}~ \tau^{1/2} \big( 1+ T_2^{\frac{1}{q_0'}} \big) \big( 1+ M_3^{2p+1} \big),
    \end{split}
    \end{equation}
where the constant $M_3 =M_3(d,p,M,\psi)> 0$ is referred to \eqref{def_M_2}.
The proof of \eqref{eq-r-53} is a minor modification of the proof in Lemma \ref{lem-K1}. Indeed, by applying Lemma \ref{lem-a-1} and \eqref{def_M_2} in Theorem \ref{wp_u_H12'} into the argument of Lemma \ref{lem-K1} instead of using \eqref{def_C_1} in Theorem \ref{wp_u_H12}, we get \eqref{eq-r-53}.

Since $T_2$ is determined by $d,p$ and $M$, we have
    $$
    \big\| Z_{\tau}^{\psi} - \Pi_{\tau} u^{\psi} \big\|_{\ell^{q} (0,T_2; W^{1,r})}
    \leq \tau^{1/2} C(d,p,M,\psi),
    $$
which completes the proof.
\end{proof}


\section{Global $H^1$ stability of $Z_{\tau}$ for $1\leq p<p_d$.}\label{sec-10}

In this section, we prove Theorem \ref{thm-9'} which is the case of $1\leq p<p_d$. The structure of the proof is to apply the local $H^1$ stability inductively after dividing the interval $(0,T]$ into a set of intervals of the same size. In doing so, the main task is to control the growth of $H^1$ norm of $Z_{\tau}$. 

\begin{proof}[Proof of Theorem \ref{thm-9'}]
We take an initial data $\phi \in H^1$, and consider a time $T>0$ such that $\sup_{0 \leq t \leq T} \|u(t)\|_{H^1 (\mathbb{R}^d)} < \infty$.
Then, from \eqref{def_C_1} in Theorem \ref{wp_u_H12}, we can find a constant $M_1= M_1 (d,p,T,\phi) \geq 1$ such that
    \begin{equation}\label{eq-5-97}
    \|u\|_{L^{\infty}(0,T;H^1)} + \|u\|_{L^q (0,T;W^{1,r})} \leq \frac{M_1}{2}.
    \end{equation}
Let us take $\beta_{d,p}>0$ in \eqref{def_T_*} and choose $T_2 >0$ as
    \begin{equation}\label{eq-5-97'}
    T_2 = \frac{\beta_{d,p}}{2} M_1^{-\frac{4p(p+2)}{4-(d-2)p}}.
    \end{equation}
If $T< 2T_2$, the stability follows just using the local stability result of Proposition \ref{thm-3}. Thus we may consider only the case $T > 2T_2$. 
\

First we consider the case  $\tau \leq T_2 /2$. As in the proof of Theorem \ref{thm-rel} we choose $R\in (T_2 /2, T_2]$ such that $R/ \tau \in \mathbb{N}$, and to argue an induction, we split $[0,T]$ as
\begin{equation}
\begin{split}
[0,T]&= \cup_{k=0}^{N-1} [kR, (k+1)R) \, \cup \, [NR, T)
\\
&=: \cup_{k=0}^{N-1} I_k \, \cup \, I_N,
\end{split}
\end{equation}
where $N \in \mathbb{N}$ is chosen so that $NR\leq T < (N+1)R$.
(We remark that $N\sim T/T_2$).

For each $k=0,1,\cdots, N-1$, take auxiliary functions $\psi_k \in H^2 (\mathbb{R}^d)$ such that
    \begin{equation}\label{def_psi}
    \| \psi_k - u^{\phi}(kR) \|_{H^1} < \frac{M_1}{(10 C_{d,p})^{N}} \, \left( < \frac{M_1}{2}\right).
    \end{equation}
Here, the constant $C_{d,p} \geq1$ is chosen as a larger one between $C_{d,p}$ given in \eqref{eq-c-1} of Theorem \ref{wp_u_H12'} and that in \eqref{eq-K-10} of Proposition \ref{prop-3-3}. We remark that we can choose a same function $\psi_k$ for any value $\tau \in (0, \bar{\tau})$ with a small $\bar{\tau} = \bar{\tau} (\phi, d,p, M_1)$ since $u$ is continuous in $H^1$ and we may choose $R \in (T_2 -2 \bar{\tau}, T_2]$ for $\tau \in (0, \bar{\tau})$.

Combining \eqref{def_psi} with \eqref{eq-5-97}, we find that
    \begin{equation}\label{eq-5-95}
    \| u(k R) \|_{H^1} \leq \frac{M_1}{2}
    \quad\textrm{and}\quad
    \|\psi_k\|_{H^1}\leq M_1
    \quad\mbox{for all}\quad k=0,\cdots,N -1.
    \end{equation}
Given \eqref{eq-5-95} and that $\psi_k \in H^2 (\mathbb{R}^d)$, we may apply Proposition \ref{prop-3-3} to obtain
    \begin{equation}\label{upper_psi_1}
    \sup_{n\tau \in [0,R]} \big\|Z_{\tau}^{\psi_k} (n\tau) - u^{\psi_k} (n\tau) \big\|_{H^1} \leq \tau^{1/2} C(d,p,M_1,\psi_k),
    \end{equation}
where $C(d,p,M_1,\psi_k)$ denotes the constant determined in \eqref{thm-7}.
Now we set a constant
    \begin{equation}\label{upper_psi_2}
    \mathbf{C}(d,p,T,\phi):=
    \frac{(10 C_{d,p})^{N}}{M_1} \max_{k=0,\cdots,N -1} C(d,p,M_1 ,\psi_k).
    \end{equation}
This is valid from $\psi_k=\psi_k(d,p,T,T_2,M_1,\phi)$, $T_2=T_2(d,p,M_1)$ and $M_1=M_1(d,p,T,\phi)$.
Then, for $\tau < \tau_* := \min\left\{ \Big( \mathbf{C}(d,p, T,\phi) \Big)^{-2}, ~\bar{\tau}\right\}$, we can deduce from
n \eqref{upper_psi_1} and \eqref{upper_psi_2} the following estimate
    \begin{equation}\label{eq-r-11}
    \sup_{n\tau \in [0,R]} \big\|Z_{\tau}^{\psi_k} (n\tau) - u^{\psi_k} (n\tau) \big\|_{H^1} \leq \frac{M_1}{(10 C_{d,p})^N}.
    \end{equation}



We shall first prove the theorem for $\tau \in (0, \tau_{*})$. To obtain the stability of $Z_{\tau}^{\phi}$ on $[0,T]$, we use the following estimates with an induction.

\medskip

\noindent \textbf{\emph{Claim}}: For any $k\in\{0,1,\cdots,N \}$, we have
    \begin{equation}\label{eq-5-91}
    \max_{n\tau \in I_{k}} \big\| Z_{\tau}^{\phi}(n\tau)  - u^{\phi} (n\tau) \big\|_{H^1}
    \leq (3 C_{d,p})^{k+1} \frac{M_1}{(10 C_{d,p})^{N}}
    \end{equation}
and
    \begin{equation}\label{eq-5-92}
    \max_{n\tau \in I_{k}} \|Z_{\tau}^{\phi}(n\tau) \|_{H^1} \leq M_1
    \end{equation}
for all $\tau \in (0,\tau_*)$.

\medskip

\noindent We prove this claim by an induction.

\noindent \textbf{Step 1.}
We show that the claim holds for $k=0$.
By the triangle inequality, we have
    $$
    \begin{aligned}
    &\max_{n\tau \in [0,R]} \big\| Z_{\tau}^{\phi}(n\tau)  - u^{\phi} (n\tau) \big\|_{H^1} \\
    &\quad\leq \max_{n\tau \in [0,R]} \Big( \big\| Z_{\tau}^{\phi}(n\tau)  - Z_{\tau}^{\psi_0}(n\tau) \big\|_{H^1} + \big\| Z_{\tau}^{\psi_0}(n\tau) - u^{\psi_0} (n\tau) \big\|_{H^1} + \big\| u^{\psi_0}(n\tau) - u^{\phi}(n\tau) \big\|_{H^1} \Big).
    \end{aligned}
    $$
Given the upper bound of initial data \eqref{eq-5-95}, we may apply \eqref{eq-K-10} in Proposition \ref{prop-3-3}, \eqref{eq-r-11}, and \eqref{eq-c-1} in Theorem \ref{wp_u_H12'} to yield that
    $$
    \begin{aligned}
    \max_{n\tau \in [0,R]} \big\| Z_{\tau}^{\phi}(n\tau)  -u^{\phi} (n\tau) \big\|_{H^1} &\leq C_{d,p} \| \phi - \psi_0 \|_{H^1} + \frac{M_1}{(10 C_{d,p})^{N}} + C_{d,p} \| \phi_0 - \psi \|_{H^1}   \\
    &\leq 2C_{d,p} \frac{M_1}{(10 C_{d,p})^{N}} + \frac{M_1}{(10 C_{d,p})^{N}} \\
    &\leq 3C_{d,p} \frac{M_1}{(10 C_{d,p})^{N}} .
    \end{aligned}
    $$
By combining this with \eqref{eq-5-97}, we obtain that
    $$
    \begin{aligned}
    \max_{n\tau \in [0,R]} \|Z_{\tau}^{\phi}(n\tau) \|_{H^1}
    &\leq \max_{n\tau \in [0,R]}  \| u^{\phi}(n\tau) \|_{H^1}
    + \max_{n\tau \in [0,R]} \big\| u^{\phi} (n\tau) - Z_{\tau}^{\phi}(n\tau) \big\|_{H^1} \\
    &\leq \frac{M_1}{2}  +  3C_{d,p} \frac{M_1}{(10 C_{d,p})^{N}}
    \leq M_1 .
    \end{aligned}
    $$
Therefore, \eqref{eq-5-92} and \eqref{eq-5-91} hold for $k=0$.

\medskip

\noindent \textbf{Step 2.} Suppose that \eqref{eq-5-91} and \eqref{eq-5-92} hold for some  $k-1 \in \{0,1,\cdots, N -1\}$.
Then we aim to show that \eqref{eq-5-91} and \eqref{eq-5-92} also hold for the $(k+1)$-step. We have
    \begin{equation}\label{eq-5-96}
    \max_{n\tau \in I_{k}} \big\| Z_{\tau}^{\phi}(n\tau) - u^{\phi}(n\tau) \big\|_{H^1}
    =\max_{n\tau \in [0,R]} \big\| Z_{\tau}^{\phi}(n \tau+ n_k \tau) - u^{\phi} (n\tau + n_k \tau) \big\|_{H^1},
    \end{equation}
    where for $k=N$, we assume that the interval $[0,R]$ is regarded as $[0, T-NR]$,  abusing a notation for the simplicity.
By the triangle inequality, we estimate the right hand side as
    \begin{equation}\label{eq-r-2}
    \begin{split}
    &\max_{n\tau \in [0,R]} \big\| Z_{\tau}^{\phi}(n \tau+ n_k \tau) - u^{\phi} (n\tau + n_k \tau) \big\|_{H^1} \\
    &~\leq \max_{n\tau \in [0,R]} \Bigl(\big\|Z_{\tau}^{\phi}(n\tau + n_k \tau) - Z_{\tau}^{\psi_k} (n\tau) \big\|_{H^1} \\
    &\qquad\qquad\qquad\qquad + \big\|Z_{\tau}^{\psi_k}(n\tau) - u^{\psi_k} (n\tau)\big\|_{H^1} + \big\|u^{\psi_k} (n\tau) - u^{\phi}(n_k \tau + n\tau)\big\|_{H^1}\Bigr).
    \end{split}
    \end{equation}
Given the estimates of initial data \eqref{eq-5-95} and \eqref{eq-5-92} with $k-1$, we may apply Lemma \ref{prop-3-3} to yield that
    $$
    \begin{aligned}
   \max_{n\tau \in [0,R]}  \big\|Z_{\tau}^{\phi} (n\tau +n_k \tau)  - Z_{\tau}^{\psi_k} (n\tau) \big\|_{H^1}& = \max_{n\tau \in [0,R]} \big\|Z_{\tau}^{Z_{\tau}^{\phi}(n_k \tau)} (n \tau) - Z_{\tau}^{\psi_k} (n\tau) \big\|_{H^1}
    \\
    & \leq C_{d,p} \big\|Z_{\tau}^{\phi}(n_k \tau) - \psi_k \big\|_{H^1}
    \\
    & \leq C_{d,p} \Bigl( \big\|Z_{\tau}^{\phi}(n_k \tau) - u^{\phi}(n_k \tau) \big\|_{H^1} + \big\|u^{\phi}(n_k \tau) - \psi_k \big\|_{H^1} \Bigr).
    \end{aligned}
    $$
Also, by Theorem \ref{wp_u_H12'}, we have
    $$
    \begin{aligned}
   \sup_{n\tau \in [0,R]} \big\|  u^{\psi_k} (n\tau) - u^{\phi} (n_k \tau +n\tau) \big\|_{H^1} &= \sup_{n\tau \in [0,R]} \big\|  u^{\psi_k} (n\tau) - u^{u^{\phi} (n_k \tau)} (n\tau) \big\|_{H^1} \\
    &\leq C_{d,p} \big\|\psi_k - u^{\phi}(n_k \tau) \big\|_{H^1}
    \end{aligned}
    $$
By inserting the estimates above into \eqref{eq-r-2} and applying \eqref{eq-r-11}, we arrive at the following estimate
    $$
    \begin{aligned}
    &\sup_{n \tau \in [0,R]} \big\|Z_{\tau}^{\phi}(n \tau  + n_k \tau) - u^{\phi} (n_k \tau  + n \tau) \big\|_{H^1} \\
    &\quad \leq C_{d,p} \big\|Z_{\tau}^{\phi}(n_k \tau) - u^{\phi}(n_k \tau) \big\|_{H^1}
    + 2C_{d,p} \Bigl( \big\|\psi_k -u^{\phi}(n_k \tau) \big\|_{H^1}\Bigr)
    + 2C_{d,p} \frac{M_1}{(10 C_{d,p})^{N}} \\
    &\quad\leq C_{d,p} \big\|Z_{\tau}^{\phi}(n_k \tau) - u^{\phi}(n_k \tau) \big\|_{H^1}  + 6C_{d,p} \frac{M_1}{(10 C_{d,p})^{N}},
    \end{aligned}
    $$
where we used \eqref{def_psi} for the second inequality.
Now we apply the assumption \eqref{eq-5-91} with the $k$-step along with \eqref{eq-5-92} in the above inequality. Then we get
    $$
    \begin{aligned}
    &\sup_{n \tau \in I_{k}} \big\|Z_{\tau}^{\phi}(n \tau  + n_k \tau) - u^{\phi} (n_k \tau  + n \tau) \big\|_{H^1} \\
    &\quad\leq C_{d,p} \bigg( (3C_{d,p})^{k+1} \frac{M_1}{(10 C_{d,p})^N} \bigg)  + 6C_{d,p} \frac{M_1}{(10 C_{d,p})^{N}} \\
    &\quad \leq (3C_{d,p})^{k+2} \frac{M_1}{(10 C_{d,p})^{N}}.
    \end{aligned}
    $$
Thus, we obtain the estimate \eqref{eq-5-91} for the $k$-step. This, together with \eqref{eq-5-97}, implies \eqref{eq-5-92}. Hence the claim is proved, and so  we have
    \begin{equation}\label{eq-r-93}
    \sup_{n \tau \in [0,T]} \big\|Z_{\tau}^{\phi}(n\tau)\big\|_{H^1} \leq M_1
    \quad\mbox{for all}\quad \tau \in (0, \tau_*),
    \end{equation}
which proves the theorem for $(q,r) = (\infty,2)$. Given estimate \eqref{eq-r-93}, the stability of $Z_{\tau}$ for general pair follows from the local $H^1$ stability result of Proposition \ref{thm-3} (see also Remark \ref{rem-r-4}). The proof is finished for $\tau \in (0, \tau_*)$.

Now, it only remains to consider the case $\tau_* \leq \tau < 1$. By H\"older's inequality and Lemma \ref{lem-2-6}, we have
    \begin{equation*}
    \begin{aligned}
    \|Z_{\tau}^{\phi} (n \tau)\|_{\ell^{q} (n\tau \in [0,T]; W^{1,r})}
    &\leq \Big( \frac{T+1}{\tau}\Big)^{\frac{1}{q}} \sup_{n\tau \in [0,T]} \|Z_{\tau}^{\phi}(n\tau)\|_{W^{1,r}} \\
    &\leq C\Big( \frac{T+1}{\tau}\Big)^{\frac{1}{q}} \big(1+ \tau^{-\frac{1}{2}}\big) \tau^{\frac{d}{2} \left( \frac{1}{r} - \frac{1}{2}\right)} \sup_{n\tau \in [0,T]} \|Z_{\tau}^{\phi}(n\tau)\|_{L^2} \\
    &\leq 2C (T+1)^{\frac{1}{2}} \big(\frac{1}{\tau_*}\big)^{\frac{1}{2}} \sup_{n\tau \in [0,T]} \|Z_{\tau}^{\phi}(n\tau)\|_{L^2}.
    \end{aligned}
    \end{equation*}
where $(q,r)$ is any admissible pair.
On the other hand, we know that the $L^2$ norm of $Z_{\tau}$ does not increase by the definition \eqref{eq-r-30}, i.e.,
    \begin{equation*}
    \|Z_{\tau}^{\phi}(n\tau)\|_{L^2}\leq \| \phi \|_{L^2} \quad\mbox{for all}\quad n \in \mathbb{N}.
    \end{equation*}
Thus we have
    \begin{equation*}
    \begin{aligned}
    \|Z_{\tau}^{\phi} (n \tau)\|_{\ell^{q} (n\tau \in [0,T]; W^{1,r})}
    \leq 2C (T+1)^{\frac{1}{2}} \big( \frac{1}{\tau_*}\big)^{\frac{1}{2}} \| \phi \|_{L^2}
    \leq C(d,p,T,\phi),
    \end{aligned}
    \end{equation*}
which proves the theorem for $\tau_* \leq \tau <1$.
\end{proof}


\section*{Acknowledgment}
This work was supported by Incheon National University Research Grant in 2017.


\end{document}